\documentclass[final,reqno]{siamltex}
\usepackage{latexsym,amsmath,amssymb,amsfonts,mathrsfs}
\usepackage{epsf,graphicx,epsfig,color,cite,cases}
\usepackage{subfigure,graphics,multirow,marginnote,graphicx}
\usepackage{bm}
\usepackage[normalem]{ulem}
\usepackage{cancel}
\newcommand{\R}{{\mathbb R}}

\newcommand{\N}{{\mathbb N}}

\newcommand{\Sp}{{\mathbb S}}

\newcommand{\no}{\nonumber}
\newcommand{\be}{\begin{eqnarray}}
\newcommand{\ben}{\begin{eqnarray*}}
\newcommand{\en}{\end{eqnarray}}
\newcommand{\enn}{\end{eqnarray*}}
\newcommand{\ba}{\backslash}
\newcommand{\pa}{\partial}

\newcommand{\ov}{\overline}
\newcommand{\I}{{\rm Im}}
\newcommand{\Rt}{{\rm Re}}

\newcommand{\G}{\Gamma}

\newcommand{\vep}{\varepsilon}
\newcommand{\Om}{\Omega}
\newcommand{\om}{\omega}

\newcommand{\la}{\lambda}

\newcommand{\ol}{\overline}

\newtheorem{remark}[theorem]{Remark}
\newtheorem{algorithm}{Algorithm}[section]
\newtheorem{problem}[theorem]{Problem}

\begin{document}
\renewcommand{\theequation}{\arabic{section}.\arabic{equation}}
%
\title{\bf Near-field imaging of an unbounded elastic rough surface with a direct imaging method
}
\author{Xiaoli Liu\thanks{Academy of Mathematics and Systems Science, Chinese Academy of Sciences,
Beijing 100190, China and School of Mathematical Sciences, University of Chinese
Academy of Sciences, Beijing 100049, China ({\tt liuxiaoli@amss.ac.cn})}
\and Bo Zhang\thanks{LSEC, NCMIS and Academy of Mathematics and Systems Science, Chinese Academy of
Sciences, Beijing, 100190, China and School of Mathematical Sciences, University of Chinese
Academy of Sciences, Beijing 100049, China ({\tt b.zhang@amt.ac.cn})}
\and Haiwen Zhang\thanks{NCMIS and Academy of Mathematics and Systems Science, Chinese Academy of Sciences,
Beijing 100190, China ({\tt zhanghaiwen@amss.ac.cn})}}
\date{}

\maketitle

\begin{abstract}
This paper is concerned with the inverse scattering problem of time-harmonic elastic waves by an
unbounded rigid rough surface. A direct imaging method is developed to reconstruct the unbounded
rough surface from the elastic scattered near-field Cauchy data generated by point sources.
A Helmholtz-Kirchhoff-type identity is derived and then used to provide a theoretical analysis
of the direct imaging algorithm.
Numerical experiments are presented to show that the direct imaging algorithm
is fast, accurate and robust with respect to noise in the data.
\end{abstract}



\pagestyle{myheadings}
\thispagestyle{plain}
\markboth{X. Liu, B. Zhang, and H. Zhang}{Near-field imaging of an unbounded elastic rough surface}

\section{Introduction}\label{se1}
\setcounter{equation}{0}

Elastic scattering problems have received much attention from both the engineering and mathematical
communities due to their significant applications in diverse scientific areas such as
seismology, remote sensing, geophysics and nondestructive testing.

This paper focuses on the inverse problem of scattering of time-harmonic elastic waves by an unbounded
rough surface. The domain above the rough surface is filled with a homogeneous and isotropic elastic medium,
and the medium below the surface is assumed to be elastically rigid.
Our purpose is to image the unbounded rough surface from the scattered elastic near-field
Cauchy data measured on a horizontal straight line segment above the rough surface and
generated by point sources. For simplicity, in this paper, we are restricted to the two-dimensional case.
See Fig. \ref{fig0} for the problem geometry.
However, our imaging algorithm can be generalized to the three-dimensional case with appropriate modifications.

\begin{figure}[htbp]
  \centering
    \includegraphics[width=4.5in]{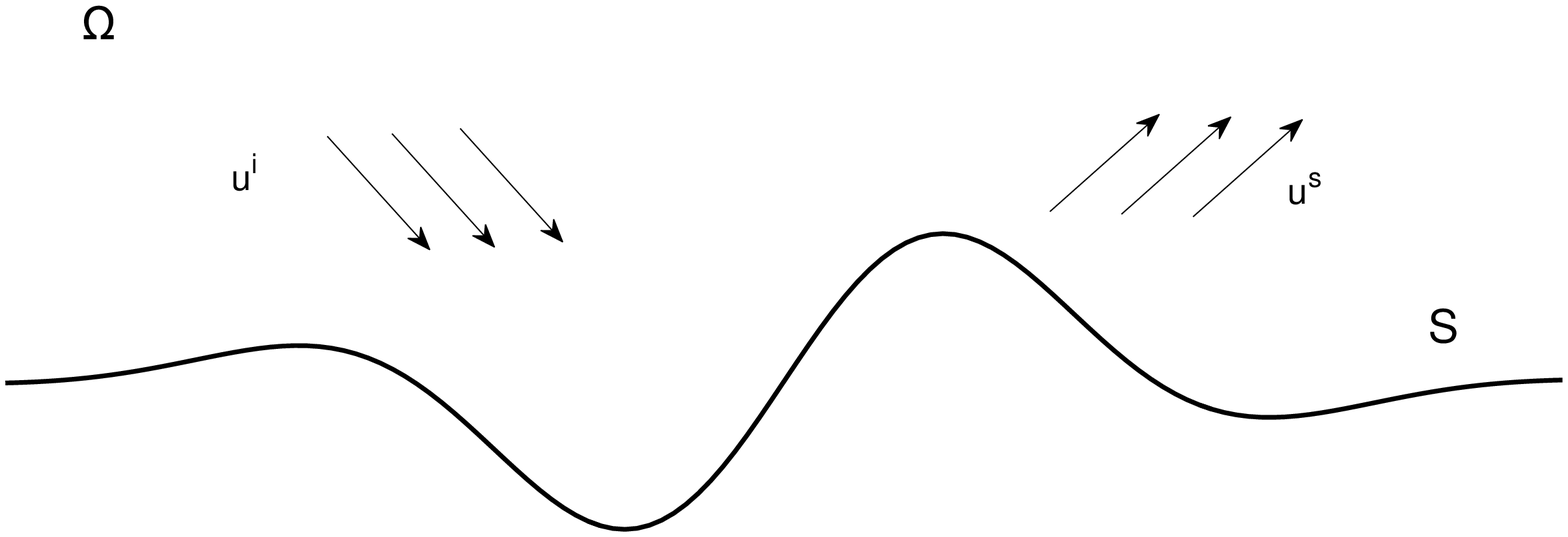}
\caption{The elastic scattering problem by a Dirichlet rough surface.
}\label{fig0}
\end{figure}

For inverse acoustic scattering problems by unbounded rough surfaces, many numerical reconstruction
methods have been developed to recover the unbounded surfaces,
such as the sampling method in \cite{AK03,YZZ12,YZZ13,YZZ16} for reconstructing periodic structures,
the time-domain probe method in \cite{BP09}, the time-domain point source method in \cite{CL05}
and the nonlinear integral equation method in \cite{LS14} for recovering sound-soft rough surfaces,
the transformed field expansion method in \cite{BL13,BL14} for reconstructing a small and smooth
deformation of a flat surface with Dirichlet, impedance or transmission conditions,
the Kirsch-Kress method in \cite{LSZ17} for a penetrable locally rough surface, the Newton
iterative method in \cite{ZZ13} for a sound-soft locally rough surface, and the direct imaging method
in \cite{LZZ18} for imaging impenetrable or penetrable unbounded rough surfaces from the scattered
near-field Cauchy data.

Compared with the acoustic case, elastic scattering problems are more complicated
due to the coexistence of the compressional and shear waves that propagate at different speeds.
This makes the elastic scattering problems more difficult to deal with theoretically and numerically.
For inverse scattering by bounded elastic bodies, several numerical inversion algorithms have been
proposed, such as the sampling method in \cite{A01a,BGC06,CKAGK07,NG04,YHXZ16}
and the transformed field expansion method in \cite{LW15,LWZ15,LWWZ16}.
See also the monograph \cite{ABG15} for a good survey.
However, as far as we know, not many results are available for inverse elastic scattering by unbounded
rough surfaces. An optimization method was proposed in \cite{EH12a} to reconstruct an elastically rigid
periodic surface, a factorization method was developed in \cite{HLZ13} for imaging an elastically rigid
periodic surface and in \cite{HKY16} for reconstructing bi-periodic interfaces between acoustic and
elastic media, and the transformed field expansion method was given in \cite{LWZ16} for imaging an
elastically rigid bi-periodic surface.

The purpose of this paper is to develop a direct imaging method for inverse elastic scattering problems
by an unbounded rigid rough surface (on which the elastic displacement vanishes).
The main idea of a direct imaging method is to present a certain imaging function to determine whether
or not a sampling point in the sampling area is on the obstacle.
Direct imaging methods usually require very little priori information of the obstacle,
and the imaging function usually involves only the measurement data but not the solution of the scattering
problem and thus can be computed easily and fast.
This is especially important for the elastic case since the computation of the elastic scattering solution
is very time-consuming. Due to the above features, direct imaging methods have recently been studied
by many authors for inverse scattering by bounded obstacles or inhomogeneity, such as the orthogonality
sampling method \cite{P10}, the direct sampling method \cite{IJZ12,LLSS13,LLZ14,L17}, the reverse time migration
method \cite{CCH13,CH14} and the references quoted there.
Recently, a direct imaging method was proposed in \cite{LZZ18} to reconstruct unbounded rough surfaces
from the acoustic scattered near-field Cauchy data generated by acoustic point sources.
This paper is a nontrivial extension of our recent work in \cite{LZZ18} from the acoustic case to the
elastic case since the elastic case is much more complicated than the acoustic case due to the coexistence
of the compressional and shear waves that propagate at different speeds.

To understand why our direct imaging method works, a theoretical analysis is presented, based on
an integral identity concerning the fundamental solution of the Navier equation (see Lemma \ref{HK})
which is established in this paper. This integral identity is similar to the {\em Helmholtz-Kirchhoff}
identity for bounded obstacles in \cite{CH14}. In addition, a reciprocity relation (see Lemma \ref{RR})
is proved for the elastic scattered field corresponding to the unbounded rough surface.
Based on these results, the required imaging function is then proposed, which, at each sampling point,
involves only the inner products of the measured data and the fundamental solution of the Navier equation
in a homogeneous background medium. Thus, the computation cost of the imaging function is very cheap.
Numerical results are consistent with the theoretical analysis and show that our imaging method can provide
an accurate and reliable reconstruction of the unbounded rough surface even in the presence of a
fairly large amount of noise in the measured data.

The remaining part of the paper is organized as follows.
In Section \ref{sec2}, we formulate the elastic scattering problem and present some inequalities
that will be used in this paper.
Moreover, the well-posedness of the forward scattering problem is presented,
based on the integral equation method (see \cite{A00,A01,A02}).
Section \ref{sec3} provides a theoretical analysis of the continuous imaging function.
Based on this, a novel direct imaging method is proposed for the inverse elastic scattering problem.
Numerical examples are carried out in Section \ref{sec4}
to illustrate the effectiveness of the imaging method.
In the appendix, we present two lemmas which are needed for the proof of the reciprocity relation.

We conclude this section by introducing some notations used throughout the paper.
For any open set $D\subset\R^m\,\,(m\in\N)$, denote by $BC(D)$ the set of
bounded and continuous, complex-valued functions in $D$, a Banach space under the
norm $\|\varphi\|_{\infty,D}:=\sup_{x\in D}|\varphi(x)|$.
Given a function $v\in L^\infty(D)$, denote by $\pa_j v$ the (distributional) derivative
$\pa v(x)/\pa x_j$, $j=1,\cdots,m$.
Define $BC^{1,1}(D):=\{\varphi\in BC(D)~|~\pa_j\varphi\in BC(D),j=1,\cdots,m\}$, equipped with
the norm $\|\varphi\|_{1,D}:=\|\varphi\|_{\infty,D}+\sum_{j=1}^m\|\pa_j\varphi\|_{\infty,D}$.
The space of H\"{o}lder continuous functions is denoted by
$C^{0,\alpha}(\ov{D}):=\{\varphi\in C(\ov{D}):\|\varphi\|_{0,\alpha;D}\}$ with the norm
$\|\varphi\|_{0,\alpha,D}:=\|\varphi\|_{\infty;D}+\sup_{x,y\in D,x\neq y}
[|\varphi(x)-\varphi(y)|/|x-y|^\alpha]$.
Denote by $C^{1,\alpha}(\R)$ the space of uniformly H\"{o}lder continuously differentiable functions
or the space of differentiable functions $\varphi$ for which $\nabla\varphi$ belongs to $C^{0,\alpha}(\R)$
with the norm $\|\varphi\|_{1,\alpha}:=\|\varphi\|_{\infty;D}+\|\nabla\varphi\|_{0,\alpha}$.
We will also make use of the standard Sobolev space $H^1(D)$ for any open set $D\subset\R^2$
and $H^{1/2}(\pa D)$ provided the boundary of $D$ is smooth enough.
The notations $H^1_{loc}(\pa D)$ and $H^{1/2}_{loc}(D)$ denote the functions
which are the elements of $H^{1/2}(\pa E)$ and $H^1(E)$ for any $E\subset\subset D$, respectively.

\section{Problem formulation}\label{sec2}
\setcounter{equation}{0}

In this section, we formulate the elastic scattering problem considered in this paper and present
its well-posedness results. Useful notations and inequalities used in the paper will also be given.

\subsection{The forward elastic scattering problem}

The propagation of time-harmonic waves with circular frequency $\omega$ in an elastic solid
with Lam\'e constants $\mu,\lambda ~(\mu >0,\lambda+\mu\geq 0)$ is governed by the Navier equation
\begin{equation}\label{naviereq}
\Delta^*{\bm u}+\omega^2{\bm u}=0.
\end{equation}
Here, ${\bm u}=(u_1,u_2)^T$ denotes the displacement field and
\ben
\Delta^*{\bm u}:=\mu\Delta{\bm u} +(\lambda+\mu)\text{grad div}\,{\bm u}.
\enn
We consider a two-dimensional unbounded rough surface $S:=\{x=(x_1, x_2)\in\R^2\;|\;x_2=f(x_1)\}$,
where $S$ is the surface profile function satisfying that $f\in BC^{1,1}(\R)$.
Denote by $\Om$ the region above $S$.
Throughout the paper, the letters $h$ and $H$ will be frequently used to denote certain real numbers
satisfying that $h<\inf f=:f_-$ and $H>\sup f=:f_+$.
For any $a\in\R$, introduce the sets
\ben
U_a:=\{{x}=(x_1,x_2)\in\mathbb{R}^2: x_2>a\},\qquad
T_a:=\{{x}=(x_1,x_2)\in\mathbb{R}^2: x_2=a\}.
\enn
Let $\Gamma({x,y})$ be the free-space Green's tensor for the two-dimensional Navier equation:
\be\label{Gamma}
\Gamma({x,y}):=\frac{1}{\mu} {\bm I} \varPhi_{k_s}({x,y})+
\frac{1}{\omega^2}\nabla^T_x\nabla_x \left(\varPhi_{k_s}({x,y})-\varPhi_{k_p}({x,y})\right),
\en
with ${x,y}\in\mathbb{R}^2,~{x\neq y}$, where the scalar function $\varPhi_{k}({x,y})$ is the fundamental
solution to the two-dimensional Helmholtz equation given by
\be\label{phi}
\varPhi_k({x,y}):=\frac{i}{4}H^{(1)}_0(k|{x-y}|), \quad {x\not=y},
\en
and the compressional wave number $k_p$ and the shear wave number $k_s$ are given by
\ben
k_p:={\omega}/{\sqrt{2\mu+\lambda}},\quad k_s:={\omega}/{\sqrt\mu}.
\enn
Motivated by the form of the corresponding Green's function for acoustic wave propagation, we define
the Green's tensor $\Gamma_{D,h}({x,y})$ for the first boundary value problem of elasticity
in the half-space $U_h$ by
\ben
\Gamma_{D.h}({x,y}):=\Gamma({x,y})-\Gamma({x,y'_h})+ U({x,y}),\quad {x,y}\in U_h,\;{x\neq y}.
\enn
Here, $y'_h=(y_1,2h-y_2)$ and the matrix function $U$ is analytic. For more properties of $U$
we refer to \cite[Chapter 2.4]{A00}.

Given a curve $\Lambda\subset\mathbb{R}^2$ with the unit normal vector ${\bm n}=(n_1,n_2)$,
the generalised stress vector ${\bm P}$ on $\Lambda$ is defined by
\be\label{pu}
{\bm P\bm u}:= (\mu +\tilde\mu)\frac{\pa{\bm u}}{\pa\bm n}
+\tilde\lambda{\bm n}~\text{div}~{\bm u}-\tilde\mu~{\bm n}^{\perp}\text{div}^{\perp}{\bm u},\quad \text{div}^{\perp}:=(-\pa_2,\pa_1).
\en
Here, $\tilde\mu,\tilde\lambda$ are real numbers satisfying $\tilde\mu+\tilde\lambda=\mu+\lambda$
and ${\bm n}^{\perp}:=(-n_2,n_1)$.

If we apply the generalised stress operator ${\bm P}$ to $\Gamma$ and $\Gamma_{D,h}$,
we obtain the matrix functions $\Pi^{(1)},\Pi^{(2)},\Pi^{(1)}_{D,h}$ and $\Pi^{(2)}_{D,h}$:
\begin{align}\label{Pi1}
\Pi^{(1)}_{jk}({x,y})&:=\left({\bm P}^{({x})}\left(\Gamma_{\cdot k}({x,y})\right)\right)_j,\\ \label{Pi2}
\Pi^{(2)}_{jk}({x,y})&:=\left({\bm P}^{({y})}\left(\Gamma_{j\cdot}({x,y})\right)^T\right)_k,\\ \label{Pi3}
\Pi^{(1)}_{D,h,jk}({x,y})&:=\left({\bm P}^{({x})}\left(\Gamma_{D,h,\cdot k}({x,y})\right)\right)_j,\\ \label{Pi4}
\Pi^{(2)}_{D,h,jk}({x,y})&:=\left({\bm P}^{({y})}\left(\Gamma_{D,h,\cdot j}({x,y})\right)^T\right)_k.
\end{align}
Given the incident wave $\bm u^i$, the scattered field ${\bm u^s}=\bm u-\bm u^i$ is required to fulfill the
{\em upwards propagating radiation condition (UPRC)} (see \cite{A01}):
\begin{equation}\label{uprc}
{\bm u}^s(x) = \int_{T_H}\Pi^{(2)}_{D.H}({x,y})\bm\phi({y})ds({y}), \quad{x}\in U_H
\end{equation}
with some $\bm\phi\in [L^{\infty}(\G_H)]^2$.

We reformulate the elastic scattering problem by the unbounded rough surface $S$
as the following boundary value problem:

\begin{problem} \label{pb1}
Given a  vector field ${\bm g}\in [BC(S)\cap H^{1/2}_{loc}(S)]^2$, find a vector field
$\bm{u}^s\in[C^2(\Omega)\cap C(\ol{\Omega})\cap H^1_{loc}(\Omega)]^2$ that satisfies
\begin{description}
\item 1) the Navier equation (\ref{naviereq}) in $\Omega$,
\item 2) the Dirichlet boundary condition $\bm u^s =\bm g$ on $S$,
\item 3) the vertical growth rate condition:
$
\sup_{{x}\in\Omega}|x_2|^\beta|{\bm u}^s(x)|<\infty
$
for some $\beta\in\mathbb{R}$,
\item 4) the UPRC (\ref{uprc}).
\end{description}
\end{problem}

\begin{remark}\label{rm1}{\rm
In Problem \ref{pb1}, the boundary data is determined by the incident wave ${\bm u}^i$,
that is, ${\bm {g}:=-{\bm u}}^{i}|_S$.
}
\end{remark}

\subsection{Some useful notations}

We give some basic notations and fundamental functions that will be needed in the subsequent discussions.

The fundamental solution $\G$ defined in (\ref{Gamma}) can be decomposed
into the compressional and shear components \cite{AR80}:
\ben
\G(x,y)=\G_p(x,y)+\G_s(x,y),\quad x,y\in \R^2,\quad x\not= 0,
\enn
where
\ben
\G_p(x,y)=-\frac{1}{\mu k_s^2}\nabla^T_x\nabla^T_x\varPhi_{k_p}(x,y)\quad\text{and}
\quad\G_s(x,y)= \frac{1}{\mu k_s^2}(k_s^2\bm I+\nabla^T_x\nabla^T_x)\varPhi_{k_s}(x,y).
\enn
The asymptotic behavior of these functions have been given in \cite[Proposition 1.14]{ABG15}:
\begin{align}\label{aaaa}
&\G_p(x,y)=\frac{k_p^2}{\om}\varPhi_{k_p}(x,y)\widehat{{y-x}}\otimes
         \widehat{{y-x}}+o\left({|x-y|}^{-1}\right),\\ \label{bbbb}
&\G_s(x,y)=\frac{k_s^2}{\om}\varPhi_{k_s}(x,y)\left(\bm I-\widehat{{y-x}}\otimes
          \widehat{{y-x}}\right)+o\left({|x-y|}^{-1}\right),\\ \label{cccc}
&\Pi^{(1)}_{\alpha}(x,y)\bm q= i\omega^2/k_{\alpha}\G_{\alpha}(x,y)\bm q + o\left({|x-y|}^{-1}\right),
\end{align}
where $\alpha=s,p$, $\Pi^{(1)}_\alpha$ is defined in (\ref{Pi1}) with $\bm n(x)={(x-y)}/{|x-y|}$
and for any vector $\bm v=(v_1,v_2)^T,$
\be\label{otimes}
\bm v\otimes \bm v = \bm v\bm v^T = \left[
\begin{matrix}
v_1^2&v_1v_2\\
v_2v_1&v_2^2
\end{matrix}
\right].
\en
From \cite[(2.14) and (2.29)]{A00} there exists a constant $C>0$ such that for ${x,y}\in U_h,$
\begin{align}\label{Gammajks}
\max\limits_{j,k=1,2}|\Gamma_{jk}({ x,y})| \leq C(1+|\log |{ x-y}|),\\ \label{Gammadhjks}
\max\limits_{j,k=1,2}|\Gamma_{D,h,jk}({ x,y})|\leq C(1+|\log |{ x-y}|).
\end{align}
Further, we have the following asymptotic property of the Hankel functions and their derivatives \cite{CK}:
For fixed $n\in\mathbb{N}$,
\be\label{bessel}
H^{(1)}_n(t)&=&\sqrt{\frac{2}{\pi t}}e^{ i(t-\frac{n\pi}{2}-\frac{\pi}{4})}
\left\{1+ O\left(\frac{1}{t}\right)\right\},\quad t\to \infty, \\ \label{bessel-}
H^{(1)'}_n(t)&=&\sqrt{\frac{2}{\pi t}}e^{i(t-\frac{n\pi}{2}+\frac{\pi}{4})}
\left\{1+ O\left(\frac{1}{t}\right)\right\},\quad t\to\infty.
\en
By (\ref{bessel}) and (\ref{bessel-}), we can deduce that
\be\label{asyGamma}
\G({x,y})\leq C(k_s,k_p,\delta)\left[\begin{matrix}
|{ x-y}|^{-\frac{1}{2}}&|{ x-y}|^{-\frac{3}{2}}\\
|{ x-y}|^{-\frac{3}{2}}&|{ x-y}|^{-\frac{1}{2}}
\end{matrix}\right]
\quad \text{as}\quad |x_1|\rightarrow\infty
\en
for ${x,y}\in U_h$ with $|x_1-y_1|>\delta$, where $\delta>0$ is a constant.
Combining this with (\ref{pu})-(\ref{Pi2}) gives
\be\label{asyPi}
\Pi^{(1)}({x,y})\leq C(k_s,k_p,\delta)\left[\begin{matrix}
|{x-y}|^{-\frac{3}{2}}&|{x-y}|^{-\frac{1}{2}}\\
|{x-y}|^{-\frac{1}{2}}&|{x-y}|^{-\frac{3}{2}}
\end{matrix}\right]
\quad \text{as}\quad |x_1|\rightarrow\infty.
\en
Further, we also need the following inequalities (see \cite[Theorems 2.13 and 2.16 (a)]{A00}:
\be\label{asyGammadh}
&&\max\limits_{j,k=1,2}|\G_{D,h,jk}(x,y)|
   \le\frac{\mathscr{H}(x_2-h,y_2-h)}{|x_1-y_1|^{3/2}},\\ \label{asyGammadh1}
&&\max\limits_{j,k=1,2}|\Pi^{(1)}_{D,h,jk}(x,y)|
   \le\frac{\mathscr{H}(x_2-h,y_2-h)}{|x_1-y_1|^{3/2}},\\ \label{asyGammadh2}
&&\max\limits_{j,k=1,2}|\Pi^{(2)}_{D,h,jk}(x,y)|
   \le\frac{\mathscr{H}(x_2-h,y_2-h)}{|x_1-y_1|^{3/2}},\\ \label{asyGammadhx}
&&\sup_{x\in T_H}\int_{S}\max\limits_{j,k=1,2}|\mathscr{G}_{jk}(x,y)|ds(y)<\infty
\en
for $|x_1-y_1|>\delta>0$, where $\mathscr{G}_{jk}$ denotes any derivative with respect to $x$
of $\Pi^{(2)}_{D,h,jk}$ and $\mathscr{H}\in C(\R^2)$.
For detailed properties of $\G$, $\G_{D,h}$, $\Pi^{(i)}$ and $\Pi^{(i)}_{D,h},~i=1,2$ can be found
in \cite[Section 2.3-2.4]{A00}.

\subsection{Well-posedness of the forward scattering problem}

The well-posedness of the elastic scattering problem reformulated as Problem \ref{pb1} has been studied
in \cite{A00,A01,A02} by the integral equation method or in \cite{EH12,EH15} by a variational approach.
We present the well-posedness results obtained in \cite{A00,A01,A02} by the integral equation method,
which will be needed in this paper. To this end, we introduce the elastic layer potentials.
For a vector-valued density ${\bm\varphi}\in[BC(S)]^2$, we define the elastic single-layer potential by
\be\label{V}
{\bm{V\varphi}(x)}:=\int_{S}\G_{D,h}(x,y){\bm\varphi}(y)ds(y)\quad{x}\in U_h\ba S,
\en
and the elastic double-layer potential by
\be\label{W}
{\bm{W\varphi}(x)}:=\int_{S}\Pi^{(2)}_{D,h}(x,y){\bm\varphi}(y)ds(y)\quad {x}\in U_h\ba S.
\en
From the properties of $\G_{D,h}$ and $\Pi^{(2)}_{D,h}$ (see \cite[Chapter 2]{A00}) it follows
that the above integrals exist as improper integrals. Further, it is easy to verify that
the potentials $\bm{V\varphi}$ and $\bm{W\varphi}$ are solutions to the Navier equation in both
$U_h\backslash\ov{\Omega}$ and $\Omega$.
We now seek a solution to Problem \ref{pb1} in the form of a combined single- and double-layer potential
\be\label{ux}
{\bm u}^s(x)=\int_S\left\{\Pi^{(2)}_{D,h}(x,y)-i\eta\G_{D,h}(x,y)\right\}{\bm\varphi}(y)ds(y),\quad{x}\in\Omega,
\en
where ${\bm\varphi}\in [BC(S)\cap H^{1/2}_{loc}(S)]^2$ and $\eta$ is a complex number with $\Rt(\eta)>0$.
From \cite[pp. 10]{A02} we know that $\bm u^s$ is a solution to the boundary value Problem \ref{pb1}
if $\bm\varphi$ is a solution to the integral equation
\be\label{eq1}
\frac{1}{2}\bm\varphi({x})+\int_S\left\{\Pi^{(2)}_{D,h}({x,y})-i\eta\G_{D,h}({x,y})\right\}\bm\varphi(y)ds(y)
=-{\bm g(x)}, \quad{x}\in S.
\en
Introduce the three integral operators ${\bm S}_f$, ${\bm D}_f$ and ${\bm D'}_f$:
\ben
{\bm S}_f\bm\phi(s)&:=&2\int_{-\infty}^{\infty}\G_{D,h}((s,f(s)),(t,f(t)))\bm\phi(t)\sqrt{1+f'(t)^2}dt, \\
{\bm D}_f\bm\phi(s)&:=&2\int_{-\infty}^{\infty}\Pi^{(2)}_{D,h}((s,f(s)),(t,f(t)))\bm\phi(t)\sqrt{1+f'(t)^2}dt,\\
{\bm D}'_f\bm\phi(s)&:=&2\int_{-\infty}^{\infty}\Pi^{(1)}_{D,h}((s,f(s)),(t,f(t)))\bm\phi(t)\sqrt{1+f'(t)^2}dt,
\enn
where $s\in\R$ and $\bm\phi\in[L^{\infty}(\R)]^2$. From \cite[Theorems 3.11 (c) and 3.12 (b)]{A00}, we know
that these three operators are bounded mappings from $[BC(\R)]^2$ into $[C^{0,\alpha}(\R)]^2$.
Set $\bm\phi(s)=\bm\varphi(s,f(s))$. Then the integral equation (\ref{eq1}) is equivalent to
\be \label{eq2}
({\bm I}+{\bm D}_f - i \eta {\bm S}_f)\phi(s)=-2{\bm g}(s,f(s)), \quad s \in \mathbb{R}.
\en
The unique solvability of (\ref{eq2}) was shown in \cite{A00,A01,A02} in both the space of bounded and continuous
functions and $L^p(\R)$, $1\le p\le\infty$, yielding existence of solutions to Problem \ref{pb1}.
We just present the conclusions that will be used in the next sections; see \cite[Chapter 5]{A00}
or \cite{A01,A02} for details.

\begin{theorem}\label{thm1} (\cite[Corollary 5.23]{A00} or \cite[Corollary 5.12]{A02})
The operator ${\bm I}+{\bm D}_f-i\eta{\bm S}_f$ is bijective on $[BC(\R)]^2$, and there holds
\ben
\|({\bm I}+{\bm D}_f - i \eta {\bm S}_f)^{-1}\|<\infty.
\enn
\end{theorem}

For the original scattering problem formulated as a boundary value problem, Problem \ref{pb1},
we have the following well-posedness result.

\begin{theorem}\label{thmwp} (\cite[Theorem 5.24]{A00} or \cite[Theorem 5.13]{A02})
For any Dirichlet data ${\bm g} \in [BC(S) \cap H^{1/2}_{loc}(S)]^2 $, there exists a unique solution
${\bm u}^s\in [C^2(\Omega) \cap C(\ol{\Omega}) \cap H^1_{loc}(\Omega)]^2$ to Problem \ref{pb1}
which depends continuously on $\|{\bm g}\|_{\infty;S}$, uniformly in $[C(\ov{\Omega}\setminus U_H)]^2$
for any $H>\sup f$.
\end{theorem}

\section{The imaging algorithm}\label{sec3}
\setcounter{equation}{0}

For $H>f_+$ and $A>0$ let $T_{H,A}:=\{x\in\Om\,|\,x_2=H,|x_1|<A\}$ be a horizontal line segment above
the rough surface $S$. Then our inverse scattering problem is to determine $S$ from the scattered near-field
Cauchy data $\{(\bm{u}^s(x;y),\bm{P}\bm{u}^s(x;y))\;|\;x ,y\in T_H\}$ corresponding to the
elastic incident point sources $\bm{u}^i(x;y,\bm{e}_j):=\Gamma(x;y)\bm{e}_j, j=1,2$.
Here, $\bm e_1=(1\;0)^T, \bm e_2=(0\;1)^T$.

We will present a novel direct imaging method to recover the unbounded rough surface $S$.
To this end, we first establish certain results for the forward scattering problem
associated with incident point sources. The imaging function for the inverse problem will then be given
at the end of this section.
In the following proofs, the constant $C>0$ may be different at different places.

\begin{lemma}\label{lem2301}
Assume that ${\bm u}^s$ is the solution to Problem \ref{pb1} with the boundary data ${\bm g}=(g_1,g_2)^T$.
If ${\bm g}\in[L^p({S})]^2$ with $1\leq p\leq\infty$,
then ${\bm u}^s,{\bm{Pu}^s}\in [L^r(T_H)]^2$ for any $H-f_+\ge\delta>0$ and ${1}/{r}<{1}/{2}+{1}/{p}$.
\end{lemma}

\begin{proof}
From (\ref{ux}), the scattered field ${\bm u}^s$ can be written in the form
\be\label{usx}
{\bm u}^s(x)=\int_S\left\{\Pi^{(2)}_{D,h}({x,y})-i\eta\G_{D,h}({x,y})\right\}
\bm\varphi(y)ds(y),\quad{x}\in\Omega,
\en
where $\bm\varphi\in[L^p({S})]^2$ is the unique solution to the boundary integral equation (\ref{eq1}).

The inequalities (\ref{asyGammadh}) and (\ref{asyGammadh2}) imply that
\ben
\max\limits_{j,k=1,2}|\Pi^{(2)}_{D,h,jk}(x,y)-i\eta\Gamma_{D,h,jk}(x,y)|
\leq\frac{\mathscr{H}(x_2-h,y_2-h)}{|x_1-y_1|^{3/2}}\in L^{{3}/{2}+\vep}(\R^2),\;\;\forall\vep>0.
\enn
This, combined with (\ref{usx}) and Young's inequality, implies that
\ben
{\bm u}^s\in [L^r(T_H)]^2\quad\text{with}\quad {1}/{r}<{1}/{2}+ {1}/{p}.
\enn
Further, by (\ref{usx}) we deduce that for ${x}\in T_H$
\ben
{\bm P^{(x)}}{\bm u}^s(x)&=&{\bm P}^{(x)}\int_S\left\{\Pi^{(2)}_{D,h}(x,y)
   -i\eta\G_{D,h}(x,y)\right\}\bm\varphi(y)ds(y),\\
&=&\int_S{\bm P}^{(x)}\left(\Pi^{(2)}_{D,h}(x,y)\varphi(y)\right)ds(y)
-i\eta\int_S\Pi_{D,h}^{(1)}(x,y)\bm\varphi(y)ds(y).
\enn
Using this equation and the inequalities (\ref{asyGammadh1}) and(\ref{asyGammadhx}) and arguing similarly
as above, we obtain that ${\bm P\bm u^s}\in [L^r(T_H)]^2$.
The lemma is thus proved.
\end{proof}

\begin{corollary}\label{asym}
For $y\in\Om$ let ${\bm u}^s(x;y,{\bm e_j})$ be the scattered field associated with the rough
surface $S$ and the incident point source ${\bm u}^i(x;y,\bm e_j):=\G(x,y)\bm e_j$ located at $y$
with polarization $\bm e_j,\;j=1,2.$ Then
$$
{\bm u}^s(\cdot;y,\bm e_j)\cdot{\ov{{\bm P\bm u}^i(\cdot;y,\bm e_j)}},\;\;
{\bm P\bm u}^s(\cdot;y,\bm e_j)\cdot{\ov{{\bm u}^i(\cdot;y,e_j)}}\in L^1(T_H),\;\;j=1,2
$$
for any $H-f_+\ge\delta>0$.
\end{corollary}

\begin{proof}
We only prove the case $j=1$. The proof of the case $j=2$ is similar.
The asymptotic property (\ref{asyGamma}) implies that
${\bm u}^i(\cdot;y,\bm e_1)\in[L^{2+\vep_1}(S)]^2\;\text{for}\;\vep_1>0.$
Thus, by Lemma \ref{lem2301} we have that
${\bm u}^s(\cdot;y,\bm e_1),\;{\bm P\bm u}^s(\cdot;y,\bm e_1)\in [L^{r}(T_H)]^2\;\text{for}\;r>1.$
From (\ref{asyGamma}) and (\ref{asyPi}), we obtain that
\ben
\ov{{\bm u}^i(\cdot;y,\bm e_1)},\;\ov{{\bm P\bm u}^i(\cdot;y,\bm e_1)}\in [L^{2+\vep_2}(T_H)]^2\;\;
\text{for}\;\;\vep_2>0.
\enn
Thus, choosing appropriate $r$ and $\vep_2$ yields
\ben
{\bm u}^s(\cdot;y,\bm e_1)\cdot\ov{{\bm P\bm u}^i(\cdot;y,\bm e_1)},\;\;
{\bm P\bm u}^s(\cdot;y,\bm e_1)\cdot\ov{{\bm u}^i(\cdot;y,\bm e_1)}\in L^{1}(T_H).
\enn
The proof is thus complete.
\end{proof}

By (\ref{phi}) and the Funk-Hecke Formula (see \cite{CK}):
\be\label{FHformula}
J_0(k|x-y|)=\frac{1}{2\pi}\int_{\mathbb{S}^1}e^{ik(x-y)\cdot\bm d}ds(\bm d),
\en
we have
\be\label{111}
\text{Im}~\varPhi_k(x,y)=\frac{1}{4}J_0(k|x-y|)=\frac{1}{8\pi}\int_{\mathbb{S}^1}e^{ik(x-y)\cdot\bm d}ds(\bm d).
\en
Taking the imaginary part of (\ref{Gamma}) and using (\ref{111}) yield
\be\no
\I\,\G(x,y)&=&\frac{1}{8\pi}\left[\frac{1}{\la+2\mu}\int_{\Sp^1}{\bm d}
  \otimes{\bm d}e^{ik_p(x-y)\cdot\bm d}ds(\bm d)\right.\\ \label{IMGamma}
&&\qquad\qquad\qquad\qquad\left.+\frac{1}{\mu}\int_{\Sp^1}(\bm I-\bm d\otimes\bm d)e^{ik_s(x-y)
\cdot\bm d}ds(\bm d)\right],
\en
where ${\bm d}=(d_1,d_2)^T\in\Sp^1:=\{x=(x_1,x_2)\;|\;|x|=1\}$ and $\otimes$ is defined as in (\ref{otimes}).

Set
\ben
\I\,\G(x,y):=\I_+\Gamma({x,y})+\I_-\Gamma({x,y}),
\enn
where
\ben
\I_+\G(x,y)&=&\frac{1}{8\pi}\left[\frac{1}{\la+2\mu}\int_{\Sp^1_+}\bm d
       \otimes{\bm d}e^{ik_p(x-y)\cdot\bm d}ds(\bm d)\right.\\
&&\qquad\qquad\qquad\qquad\left.+\frac{1}{\mu}\int_{\Sp^1_+}(\bm I-\bm d\otimes\bm d)e^{ik_s(x-y)
       \cdot\bm d}ds(\bm d)\right],\\
\I_-\G(x,y)&=&\frac{1}{8\pi}\left[\frac{1}{\la+2\mu}\int_{\Sp^1_-}\bm d
   \otimes{\bm d}e^{ik_p(x-y)\cdot \bm d}ds(\bm d)\right.\\
&&\qquad\qquad\qquad\qquad\left.+\frac{1}{\mu}\int_{\Sp^1_-}(\bm I-\bm d\otimes\bm d)e^{ik_s(x-y)
  \cdot\bm d}ds(\bm d)\right].
\enn
Then we have the following identity which is similar to the Helmholtz-Kirchhoff identity for the elastic
scattering by bounded obstacles \cite{CH14} and the acoustic scattering by rough surfaces \cite{LZZ18}.

\begin{lemma}\label{HK}
For any $H\in\R$ we have
\be\label{hhhkkk}
\int_{T_H}\left([\Pi^{(1)}(\xi,x)]^T\ov{\G(\xi,y)}-[\G(\xi,x)]^T\ov{\Pi^{(1)}(\xi,y)}\right)ds(\xi)
=2i\I_+\G(y,x)
\en
for $x,y\in\R^2\ba\ov{U}_H$, where $\Pi^{(1)}$ is defined in (\ref{Pi1}) with the unit
normal $\bm n$ on $T_{H}$ pointing into $U_H$.
\end{lemma}

\begin{proof}
Let $\pa B^+_R$ be the upper half-circle above $T_H$ centered at $(0,H)$ with radius $R>0$ and define
$T_{H,R}:=\{x\in T_H~|~|x_1|\le R\}$. Denote by $B$ the bounded region enclosed by $T_{H,R}$ and $\pa B^+_R$.
For any vectors $\bm p,\bm q\in\R^2$, $\G({\cdot,x})\bm q$ and $\ov{\G({\cdot,y})}\bm p$ satisfy the Navier
equation in $B$, so, by the third generalised Betti formula (\cite[Lemma 2.4]{A00}) we have
\begin{align*}
0&=\int_{B}\left(\Delta^*\G({\xi,x})\bm q+\om^2\G({\xi,x})\bm q\right)\cdot\ov{\G({\xi,y})}\bm p~d\xi\\
&=\int_B\left(\Delta^*\G(\xi,x)\bm q\cdot\ov{\G(\xi,y)}\bm p
   -\Delta^*\ov{\G(\xi,y)}\bm p\cdot\G(\xi,x)\bm q\right)d\xi\\
&=\int_{T_{H,R}}\left(-\ov{\G(\xi,y)}\bm p\cdot\bm P[\G(\xi,x)\bm q]
  +\G(\xi,x)\bm q\cdot\bm P[\ov{\G(\xi,y)}\bm p]\right)ds(\xi)\\
&\quad+\int_{\pa B^+_R}\left(\ov{\G({\xi,y})}\bm p\cdot\bm P[\G({\xi,x})\bm q]
  -\G({\xi,x})\bm q\cdot\bm P[\ov{\G({\xi,y})}\bm p]\right)ds(\xi),
\end{align*}
that is,
\be\no
&&\int_{T_{H,R}}\left(\Pi^{(1)}({\xi,x})\bm q\cdot\ov{\G({\xi,y})}\bm p
  -\G({\xi,x})\bm q\cdot\ov{\Pi^{(1)}({\xi,y})}\bm p\right)ds(\xi)\\ \label{eq222}
&&\qquad\qquad=\int_{\pa B^+_R}\left(\Pi^{(1)}({\xi,x})\bm q\cdot\ov{\G({\xi,y})}\bm p
-\G({\xi,x})\bm q\cdot\ov{\Pi^{(1)}({\xi,y})}\bm p\right)ds(\xi).
\en
Using the asymptotic properties (\ref{aaaa})-(\ref{cccc}), we obtain that
\ben
&&\lim_{R\rightarrow\infty}\int_{\pa B^+_R}\left(\Pi^{(1)}({\xi,x})\bm q\cdot\ov{\G({\xi,y})}\bm p
   -\G({\xi,x})\bm q\cdot\ov{\Pi^{(1)}({\xi,y})}\bm p\right)ds(\xi)\\
&&\;=\lim_{R\rightarrow\infty}\int_{\pa B^+_R}i\om\left(c_s\G_s({\xi,x})
  +c_p\G_p({\xi,x})\right)\bm q\cdot\left(\ol{\G_s({\xi,y})}+\ov{\G_p({\xi,y})}\right){\bm p}ds(\xi)\\
&&\;\;+\lim_{R\rightarrow\infty}\int_{\pa B^+_R}\left(\G_s({\xi,x})
  +\G_p({\xi,x})\right)\bm q\cdot i\om\left(c_s\ov{\G_s({\xi,y})}+c_p\ov{\G_p({\xi,y})}\right){\bm p}ds(\xi)\\
&&\;=2i\om\bm q\cdot\lim_{R\rightarrow\infty}\int_{\pa B^+_R}\left(c_s\G^T_s({\xi,x})\ov{\G_s({\xi,y})}
  +c_p\G^T_p({\xi,x})\ov{\G_p({\xi,y})}\right)ds(\xi){\bm p} \\
&&\;=\frac{i}{4\pi}\bm q\cdot\int_{\Sp^1_+}\left[\frac{1}{\mu}\left(\bm I-\hat{\xi}\hat{\xi}^T\right)
  e^{-ik_s\hat{\xi}\cdot(x-y)}+\frac{1}{\la+2\mu}\hat{\xi}\hat{\xi}^T
  e^{-ik_p\hat{\xi}\cdot(x-y)}\right]ds(\xi){\bm p}\\
&&\;=2i\bm q\cdot\I_+~\G({y,x})\bm p.
\enn
Then (\ref{hhhkkk}) follows from (\ref{eq222}). The proof is thus complete.
\end{proof}

We need the following reciprocity relation result.

\begin{lemma}\label{RR} (Reciprocity relation)
For ${\bm p}, {\bm q}\in \R^2$, let ${\bm u}^s(z;x,\bm p)$ and ${\bm u}^s(z;y,\bm q)$ be the scattered fields in $\Om$ associated with
the rough surface $S$ and the incident point sources ${\bm u}^i(z;x,\bm p):=\G(z;x)\bm p$ and
${\bm u}^i(z;y,\bm q):=\G(z;y)\bm q$, respectively. Then
\be\label{leRR}
{\bm u}^s(x;y,\bm p)\cdot\bm q={\bm u}^s(y;x,\bm q)\cdot\bm p,\;\;x,y\in\Om.
\en
\end{lemma}

\begin{proof}
For $b,L,\vep>0$ define
$$
D_{b,L,\vep}:=\{z\in\Omega\;|\;|z_2|< b,|z_1|< L,|z-x|>\vep,|z-y|>\vep\}.
$$
Choose $\vep$ sufficiently small and $b,L$ large enough such that $\ov{B_{\vep}(x)}\subset D_{b,L,\vep}$,
$\ov{B_{\vep}(y)}\subset D_{b,L,\vep}$ and $\ov{B_{\vep}(x)}\cap\ov{B_{\vep}(y)}=\emptyset$.
Then, by Theorem \ref{thmwp}, ${\bm u}\in [C^2(\Omega)\cap C(\ov{\Omega})\cap H^1_{loc}(\Omega)]^2$.
Thus, use the third generalised Betti formula (see \cite[Lemma 2.4]{A00}) in $D_{b,L,\vep}$ to get
\be\label{eq330}
0=\int_{\pa D_{b,L,\vep}}\left[({\bm P\bm u})(z;x,\bm p)\cdot{\bm u}(z;y,\bm q)
-({\bm P\bm u})(z;y,\bm q)\cdot{\bm u}(z;x,\bm p)\right] ds(z),
\en
where $\pa D_{b,L,\vep}:= T_{b,L}\cup T^\pm_{L}\cup\pa B_{\vep}({x})\cup\pa B_{\vep}({y})\cup S_L$ with
$T^\pm_{L}:=\{z\in\Omega~|~z_1=\pm L,~|z_2|< b\}$ and $S_{L}:=\{z\in S~|~|z_1|<L\}$.
From Lemma \ref{lemmaA} it follows that
\begin{equation}\label{eq332}
\lim_{L\to\infty}\int_{T_{b,L}}\left[({\bm P\bm u})(z;{x},\bm p)\cdot{\bm u}(z;{y},\bm q)
-({\bm P\bm u})(z;y,\bm q)\cdot{\bm u}(z;x,\bm p)\right]ds(z)=0.
\end{equation}
Using (\ref{asyGamma}) gives
\be\label{eq333}
\lim_{L\to\infty}\int_{T^\pm_{L}}\left[({\bm P\bm u})(z;x,\bm p)\cdot{\bm u}(z;y,\bm q)
-({\bm P\bm u})(z;y,\bm q)\cdot{\bm u}(z;x,\bm p)\right] ds(z)=0.\;
\en
Now, applying the third generalised Betti formula to ${\bm u}^s(z;y,\bm q)$ and ${\bm u}(z;x,\bm p)$
in $B_{\vep}(y)$, we have
\be\label{eq335}
\int_{\pa B_{\vep}(y)}\left[({\bm P\bm u})(z;x,\bm p)\cdot {\bm u}^s(z;y,\bm q)
-({\bm P\bm u}^s)(z;y,\bm q)\cdot {\bm u}(z;x,\bm p)\right] ds(z)=0,\;
\en
while applying the third generalised Betti formula to ${\bm u}^s(z;x,p)$ and ${\bm u}(z;y,\bm q)$
in $B_{\vep}(x)$ gives
\be\label{eq336}
\int_{\pa B_{\vep}(x)}\left[({\bm P\bm u}^s)(z;x,\bm p)\cdot {\bm u}(z;y,\bm q)
-({\bm P\bm u})(z;y,\bm q)\cdot {\bm u}^s(z;x,\bm p)\right] ds(z)=0.
\en
Letting $L\to\infty$ in (\ref{eq330}) and using (\ref{eq332})-(\ref{eq336}) and the Dirichlet boundary
condition $\bm u(z;x,\bm p)=\bm u(z;y,\bm q)=0$ on $S$ yield
\ben
0&=&\int_{\pa B_{\vep}(x)}\left[({\bm P\bm u})(z;y,\bm q)\cdot {\bm u}^i(z;x,\bm p)
-{\bm u}(z;y,\bm q)\cdot ({\bm P\bm u}^i)(z;x,\bm p)\right]ds(z)\\
&& -\int_{\pa B_{\vep}(y)}\left[({\bm P\bm u})(z;x,\bm q)\cdot {\bm u}^i(z;y,\bm p)
-{\bm u}(z;x,\bm q)\cdot ({\bm P\bm u}^i)(z;y,\bm p)\right] ds(z).
\enn
This, together with Lemma \ref{lemmaB}, implies
\ben
{\bm u}({x;y},\bm q)\cdot\bm p={\bm u}({y;x},\bm p)\cdot\bm q.
\enn
The required equation (\ref{leRR}) then follows from the above equation and the fact that
${\bm u}^i({x;y},\bm q)\cdot\bm p={\bm u}^i({y;x},\bm p)\cdot\bm q$.
The proof is thus complete.
\end{proof}

We are now ready to prove the following theorem which leads to the imaging function for our imaging algorithm.

\begin{theorem}\label{maintheorem}
Let ${\bm u}^{s}(x;y,\bm e_j)$ be the unique scattered field generated by the incident point source
${\bm u}^{i}(x;y,\bm e_j):=\G(x;y)\bm e_j$ located at $y\in\Om$, $j=1,2$.
Define
\ben
{\bm U}^{i}(y;z,\bm e_j):=2i\bm e_j^T\I~\G(y,z).
\enn
Then the scattered field generated by ${\bm U}^{i}({y;z},\bm e_j)$ is given by
\ben
{\bm U}^{s}(y;z,\bm e_j)&:=&\int_{T_H}\left({\bm P\bm u}^{s}(x;y,\bm e_j)\cdot\ov{{\bm u}^{i}(x;z,\bm e_j)}
-{\bm u}^{s}(x;y,\bm e_j)\cdot\ov{{\bm P\bm u}^{i}(x;z,\bm e_j)}\right)ds(x)\\
&&-2i\bm e_j\cdot\I_-~\G(z,y)\bm e_j,\;\;j=1,2.
\enn
\end{theorem}

\begin{proof}
Note first that, by Corollary \ref{asym} ${\bm U}^{s}({y;z},e_j)$ is well-defined for $j=1,2$.
Now define the integral operators ${\bm S}$ and ${\bm D}$ by
\ben
{\bm S}\bm\varphi({x})&:=& 2\int_{S}\G_{D,h}({x,y})\bm\varphi({y})ds({y}), \quad {x}\in S, \\
{\bm D}\bm\varphi({x})&:=& 2\int_{S}\Pi^{(2)}_{D,h}({x,y})\bm\varphi({y})ds({y}), \quad {x}\in S.
\enn
Since ${\bm u}^s({y;x},\bm e_j)$ is the solution to Problem \ref{pb1} with the boundary data
${\bm g}_{x,\bm e_j}(y)=-\Gamma({y,x})\bm e_j$, $y\in S$,
then, by Theorem \ref{thmwp} there exists $\bm\varphi_{x,\bm e_j}(y)\in[BC(S)]^2$ such that
\ben
{\bm u}^{s}({y,x};\bm e_j)=\int_S\left[\Pi^{(2)}_{D,h}({y,\xi})-i\eta\G_{D,h}({y,\xi})\right]
\bm\varphi_{x,\bm e_j}({\xi})ds({\xi}), \quad {y}\in\Omega,
\enn
where $\bm\varphi_{{x},\bm e_j}$ satisfied the integral equation
$({\bm I}+{\bm D}-i\eta{\bm S})\bm\varphi_{x,\bm e_j}=-2{\bm g}_{x,\bm e_j}$,
and the integral operator ${\bm I}+{\bm D}-i\eta{\bm S}$ is bijective (and so boundedly invertible) in $[BC(S)]^2$.
Here, we use the subscript ${\bm x}$ to indicate the dependance on the point ${\bm x}$.
Further, since, for $y\in S$ the boundary data ${\bm g}_{x,\bm e_j}(y)$ is differentiable with respect
to $x\in T_H$, then we have $\bm\varphi_{x,\bm e_j}$ is differentiable with respect to $x$ and
\be\label{biemm}
{\bm P^{(x)}}\bm\varphi_{{x},\bm e_j}=-2({\bm I}+{\bm D}- i\eta{\bm S})^{-1}{\bm P^{(x)}}{\bm g}_{x,\bm e_j}.
\en
Define $\tilde{\bm\varphi}_{{x},\bm e_j}:={\bm P^{(x)}}\bm\varphi_{{x},\bm e_j}$ and
$\tilde{{\bm g}}_{x,\bm e_j}:={\bm P^{(x)}}{\bm g}_{x,\bm e_j}$. Then (\ref{biemm}) can be rewritten as
$({\bm I}+{\bm D}- i\eta{\bm S})\tilde{\bm\varphi}_{{x},\bm e_j}=-2\tilde{\bm g}_{x,\bm e_j}$. Thus,
\be\label{1us}
\tilde{\bm u}^{s}(y;x,\bm e_j):=\int_S\left\{\Pi^{(2)}_{D,h}(y,\xi)-i\eta\G_{D,h}({y,\xi})\right\}
\tilde{\bm\varphi}_{x,\bm e_j}({\xi})ds({\xi}),\quad y\in\Om
\en
is a solution to Problem \ref{pb1} with the boundary data $\tilde{\bm g}_{x,\bm e_j}$.
Since the layer-potentials operators $\bm V$ and $\bm W$ defined by (\ref{V}) and (\ref{W}), respectively,
are bounded linear operators from $[H^{1/2}_{loc}(S)]^2$ to $[H^{1}_{loc}(U_h\ba S)]^2$ (see \cite{A00}),
we have
\be\label{eq008}
\tilde{\bm u}^{s}({y;x},\bm e_j)={\bm P}^{(x)}{\bm u}^{s}({y;x},\bm e_j),\;\;j=1,2.
\en
By (\ref{biemm}), (\ref{1us}), (\ref{eq008}) and the reciprocity relation in Lemma \ref{RR}, we obtain that
${\bm P}^{(x)}{\bm u}^{s}({x;y},\bm e_j)$ is the solution to Problem \ref{pb1}
with the boundary data ${\bm P^{(x)}}{\bm g}_{x,\bm e_j}(y)$, $j=1,2$.

Now, for $j=1,2$ define
\ben
\tilde{\tilde{\bm\varphi}}_{z,\bm e_j}(y)
&:=&\int_{T_H}\left(\tilde{\bm\varphi}_{x,\bm e_j}(y)\ov{\G({x,z})}\bm e_j
-\bm\varphi_{x,\bm e_j}(y)\ov{\Pi^{(1)}({x,z})}\bm e_j\right)ds(x), \\
\tilde{\tilde{{\bm g}}}_{z,\bm e_j}(y)
&:=&\int_{T_H}\left(\tilde{\bm g}_{x,\bm e_j}(y)\ov{\G({x,z})}\bm e_j
-{\bm g}_{x,\bm e_j}(y)\ov{\Pi^{(1)}({x,z})}\bm e_j\right)ds({x}).
\enn
By (\ref{Gammajks}), (\ref{asyGamma}) and (\ref{asyPi}), and since
\be\label{eq011}
({\bm I}+{\bm D}- i\eta{\bm S})\tilde{\tilde{\bm\varphi}}_{{z},\bm e_j}({y})
=-2\tilde{\tilde{{\bm g}}}_{z,\bm e_j}(y),\quad y\in S,
\en
then it follows that
\be\label{eq012}
\tilde{\tilde{{\bm u}}}^s({y;z},\bm e_j):=\int_S\left[\Pi^{(2)}_{D,h}({y,\xi})
-i\eta\G_{D,h}({y,\xi})\right]\tilde{\tilde{\bm\varphi}}_{z,\bm e_j}({\xi})ds({\xi}),
\quad{y}\in\Omega
\en
is the solution to Problem \ref{pb1} with the boundary data $\tilde{\tilde{\bm g}}_{z,\bm e_j}(y)$, $j=1,2$.

From the above analysis it is seen that ${\bm U}^{s}({y;z},\bm e_j)$ is the solution to Problem \ref{pb1}
with the boundary data
\ben
{\bm G}({y;z},q)&=&-\int_{T_H}\left(\Pi^{(1)}({x;y},\bm e_j)\ov{\G({x,z})}\bm e_j
-\Gamma({x,y})\bm e_j\ov{\Pi^{(1)}({x,z})}\bm e_j\right)ds({x})\\
&&\;\; -2i\bm e_j^T\I_-\G({z,y}),\quad {y}\in S.
\enn
By Lemma \ref{HK} and the {Funk-Hecke} formula (\ref{FHformula}) we deduce that
\ben
{\bm G}({y;z},\bm e_j)=-2i\bm e_j^T\I_+\G({z,y})-2i\bm e_j^T\I_-\G({z,y})=-2i\bm e_j^T\I\,\G({z,y}),\;\;j=1,2.
\enn
The proof is thus complete.
\end{proof}

By Theorems \ref{thmwp} and \ref{maintheorem}, the scattered field ${\bm U}^s({y;z},\bm e_j)$ can be
expressed as
\be\label{US}
{\bm U}^{s}({y;z},\bm e_j)=\int_S\left\{\Pi^{(2)}_{D,h}({y,\xi})-i\eta\G_{D,h}({y,\xi})\right\}
\bm\varphi_{{z},\bm e_j}({\xi})ds({\xi}),\quad {y}\in\Om,
\en
where $\bm\varphi_{z,\bm e_j}$ is the unique solution to the integral equation
\ben
({\bm I}+{\bm D}-i\eta{\bm S})\bm\varphi_{z,\bm e_j}(y)=-4i\bm e_j^T\I\,\G({y,z}),
\quad{y}\in S,\;\;j=1,2.
\enn
Since ${\bm I}+{\bm D}- i\eta{\bm S}$ is bijective (and so boundedly invertible) in $[BC(S)]^2$, we have
\be\label{eq339}
C_1\|\bm e_j^T\I\,\Gamma({\cdot,z})\|_{\infty,S}\le\|\bm e_j\varphi_{z,\bm e_j}\|_{\infty,S}
\le C_2 \|\bm e_j^T\I\,\G({\cdot,z})\|_{\infty,S}
\en
for some positive constants $C_1,C_2$. Since the Bessel functions have the asymptotic
behavior \cite[Section 2.4]{CK}:
\ben
J_n(t)&=&\sum_{p=0}^\infty\frac{(-1)^p}{p!(n+p)!}\left(\frac{t}{2}\right)^{n+2p},\quad t\in\R,\\
J_n(t)&=&\sqrt{\frac{2}{\pi t}}\cos\left(t-\frac{n\pi}{2}
-\frac{\pi}{4}\right)\left\{1+O\left(\frac{1}{t}\right)\right\},\quad t\rightarrow\infty
\enn
for $n=0,1,...$, and noting that $J'_0(t)=-J_1(t)$, we know that
\ben
\max\limits_{j,k=1,2}|\I\,\G_{jk}({y,z})|=\left\{
\begin{array}{ll}
O(1) &\textrm{if}\;y=z,\\
O\left({|y-z|^{-1/2}}\right)&\textrm{if}\;|{y-z}|>>1.
\end{array} \right.
\enn
Therefore, it follows from (\ref{eq339}) that for $j=1,2$,
\ben
\left\{\begin{array}{ll}
\|\varphi_{z,\bm e_j}\|_{\infty,S}\geq C_1 &\textrm{if}\;z\in S{\color{red},}\\
\|\varphi_{z,\bm e_j}\|_{\infty,S}=O\left({d(z,\Gamma)^{-1/2}}\right)&\textrm{if}\;d(z,\G)>>1.
\end{array} \right.
\enn
From this and (\ref{US}), it is expected that the scattered field ${\bm U}^s({y;z},\bm e_j)$
takes a large value when $z\in S$ and decays as $z$ moves away from $S$.
Thus, if we choose the imaging function
\be\no
I(z)&:=&{\sum^2_{j=1}}\int_{T_H}\left|\int_{T_H}\left({\bm P\bm u}^{s}(x;y,\bm e_j)
\cdot\ov{{\bm u}^{i}(x;z,\bm e_j)}
- {\bm u}^{s}(x;y,\bm e_j)\cdot\ov{{\bm P\bm u}^{i}(x;z,\bm e_j)}\right)ds(x)\right.\\ \label{indicator}
&&\quad\left.-2i\bm e_j^T\I_-\G(z,y)\bm e_j\right|^2ds(y),
\en
then it is expected that $I(z)$ takes a large value when $z\in S$ and decays as $z$ moves away from
the rough surface $S$. Based on (\ref{indicator}), we can develop a direct imaging algorithm to
locate the unbounded rough surface $S$ by only computing the inner products of the measured data
($\bm u^s,\bm P\bm u^s$) with the fundamental solution in the homogeneous background at each sampling point.

In numerical computations, the infinite integration interval $T_H$ in (\ref{indicator}) is truncated to be
a finite one ${T_{H,A}}:=\{{x}\in T_H~|~|x_1|<A\}$ which is then discretized uniformly into $2N$
subintervals with the step size $h=A/N$. In addition, the lower half-circle $\Sp^1_-$ in the integral
of $\I_-\G(z,y)$ in (\ref{indicator}) is also uniformly discretized into $M$ grids with the step
size $\Delta\theta=\pi/M$. Then for each sampling point $z$ we get the following discrete form of
the imaging function (\ref{indicator})
\be\label{eq188}
I_A(z)=\sum^2_{j=1}h\sum_{k=0}^{2N}\left|I_{1,A,j,k}(z)-I_{2,A,j,k}(z)\right|^2
\en
with
\ben
I_{1,A,j,k}(z)&=&h\sum_{i=0}^{2N}\left({\bm P\bm u}^{s}(x_i;y_k,\bm e_j)\cdot\ov{{\bm u}^{i}(x_i;z,\bm e_j)}
-{\bm u}^{s}(x_i;y_k,\bm e_j)\cdot\ov{{\bm P\bm u}^{i}(x_i;z,\bm e_j)}\right),\\
I_{2,A,j,k}(z)&=&\frac{1}{8\pi}\left[\frac{\Delta\theta}{\lambda+2\mu}
\sum_{k=0}^{2M}{\bm d}_k\otimes{\bm d}_k e^{ik_p(z-y_j)\cdot\bm d_k}+\frac{\Delta\theta}{\mu}
\sum_{k=0}^{2M}(\bm I-\bm d_k\otimes\bm d_k)e^{ik_s(z-y_j)\cdot\bm d_k}\right].
\enn
Here, the measurement points are denoted by $x_i=(-A+i\cdot h,H),~i=0,1,...,2N$,
the incident point source positions are $y_k=(-A+k\cdot h,H),~k=0,1,...,2N$ and
the directions $\bm d_k=(\sin(-\pi+k\Delta\theta),\cos(-\pi+k\Delta\theta)),~k=0,1,...,2M$.

The direct imaging method based on (\ref{eq188}) can be presented in the following algorithm.

\begin{algorithm}\label{alg}
Let $K$ be the sampling region which contains the part of the rough surface that we hope to recover.
\begin{description}
\item 1) Choose $\mathcal{T}_m$ to be a mesh of $K$ and $T_{H,A}$ to be a straight line segment
         above the rough surface.
\item 2) Collect the Cauchy data $({\bm u}^s(x_i;y_k,\bm e_j), {\bm P\bm u}^s(x_i;y_k,\bm e_j)),~i,k =0,1,...,2N$ on $T_{H,A}$ corresponding to the incident point source located at ${y}_k$.
\item 3) For all sampling points $z\in\mathcal{T}_m$, compute the discrete imaging function $I_A(z)$
         in (\ref{eq188}).
\item 4) Locate all those sampling points $z\in\mathcal{T}_m$ such that $I_A(z)$ takes a large value,
which represent the part of the rough surface on the sampling region $K$.
\end{description}
\end{algorithm}

\section{Numerical Experiments}\label{sec4}
\setcounter{equation}{0}

We now present several numerical experiments to demonstrate the effectiveness of our imaging algorithm
and compare the reconstructed results under different parameters.
To generate the synthetic data, we use the Nystr\"{o}m method to solve the forward scattering
problem based on the integral equation technique \cite{LSZZ18}.
The noisy Cauchy data are generated as follows
\ben
\bm u^s_{\delta}(x)&=&\bm u^s(x)+\delta(\zeta_1+i\zeta_2)\max_x|\bm u^s(x)|,\\
{\bm P u}^s_{\delta}(x)&=&\pa_\nu\bm u^s(x)+\delta(\zeta_1+i\zeta_2)\max_x|{\bm P\bm u}^s(x)|,
\enn
where $\delta$ is the noise ratio and $\zeta_1,\zeta_2$ are the standard normal distributions.
In all examples, we choose $N=100$ and $M=256$.
The sampling region will be set to be a rectangular domain and
the Lam\'e constants are taken as $\mu = 1, \lambda = 1$.
In each figure, we use a solid line to represent the actual rough surface.

{\bf Example 1.} In this example we compare the reconstructed results with different wave numbers.
The exact rough surface $S$ is given by
\ben
f(x_1)= 0.5+0.03\sin[2.5\pi(x_1-1)]+0.12\sin[0.4\pi(x_1-1)]
\enn
The scattered near-field Cauchy data is measured on $\G_{H,A}$ with $H=2,~A=20$ and polluted with 20\% noise.
The reconstructed results show that the macro-scale features of the rough surface are captured
with smaller wave numbers (see Fig. \ref{fig1} (a)) and the whole rough surface is accurately recovered
with larger wave numbers (see Fig. \ref{fig1} (c)).

\begin{figure}[htbp]
  \centering
  \subfigure[\textbf{$k_p=5.20,\;k_s=9$}]{\includegraphics[width=1.65in]{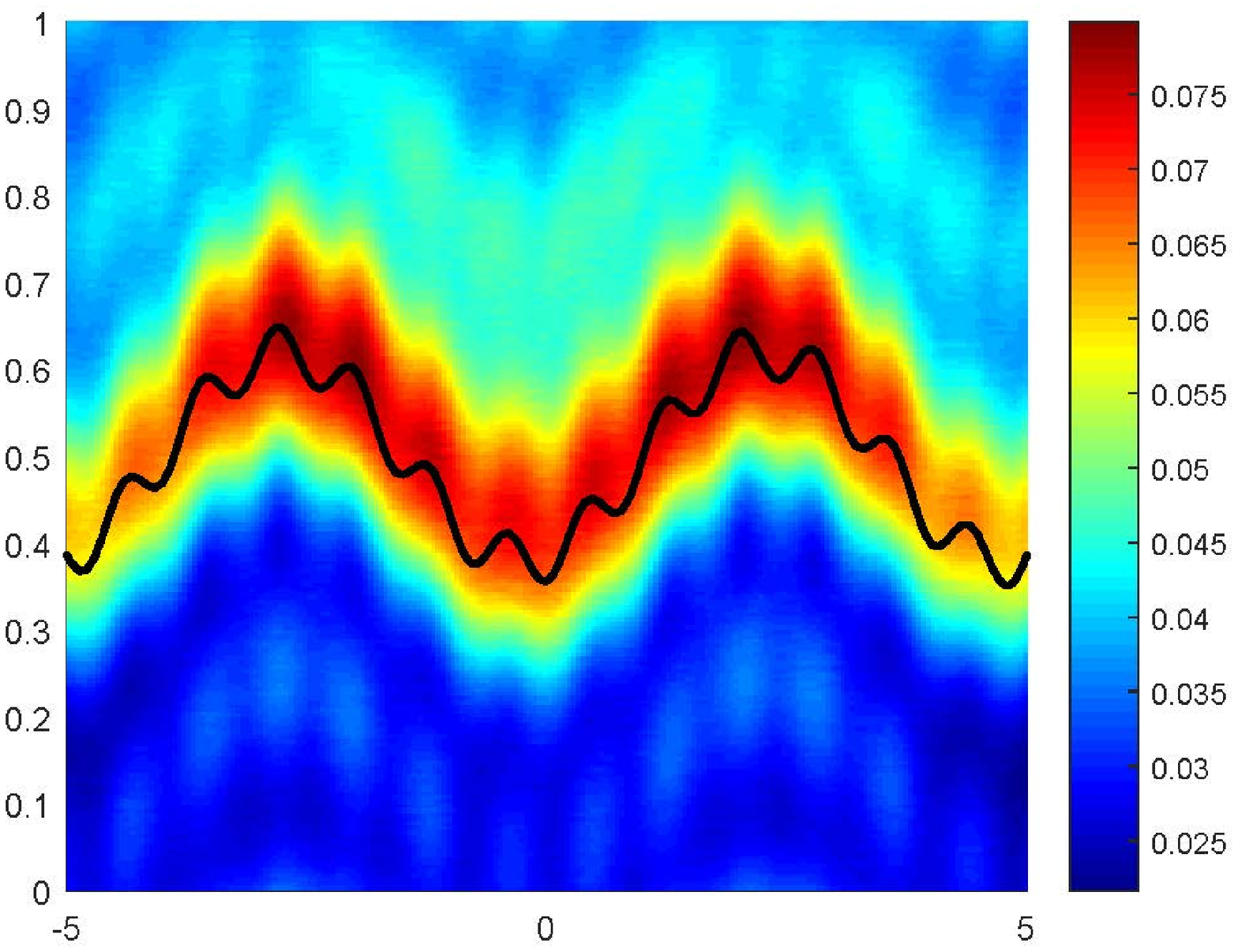}}
  \subfigure[\textbf{$k_p=8.66,\;k_s=15$}]{\includegraphics[width=1.65in]{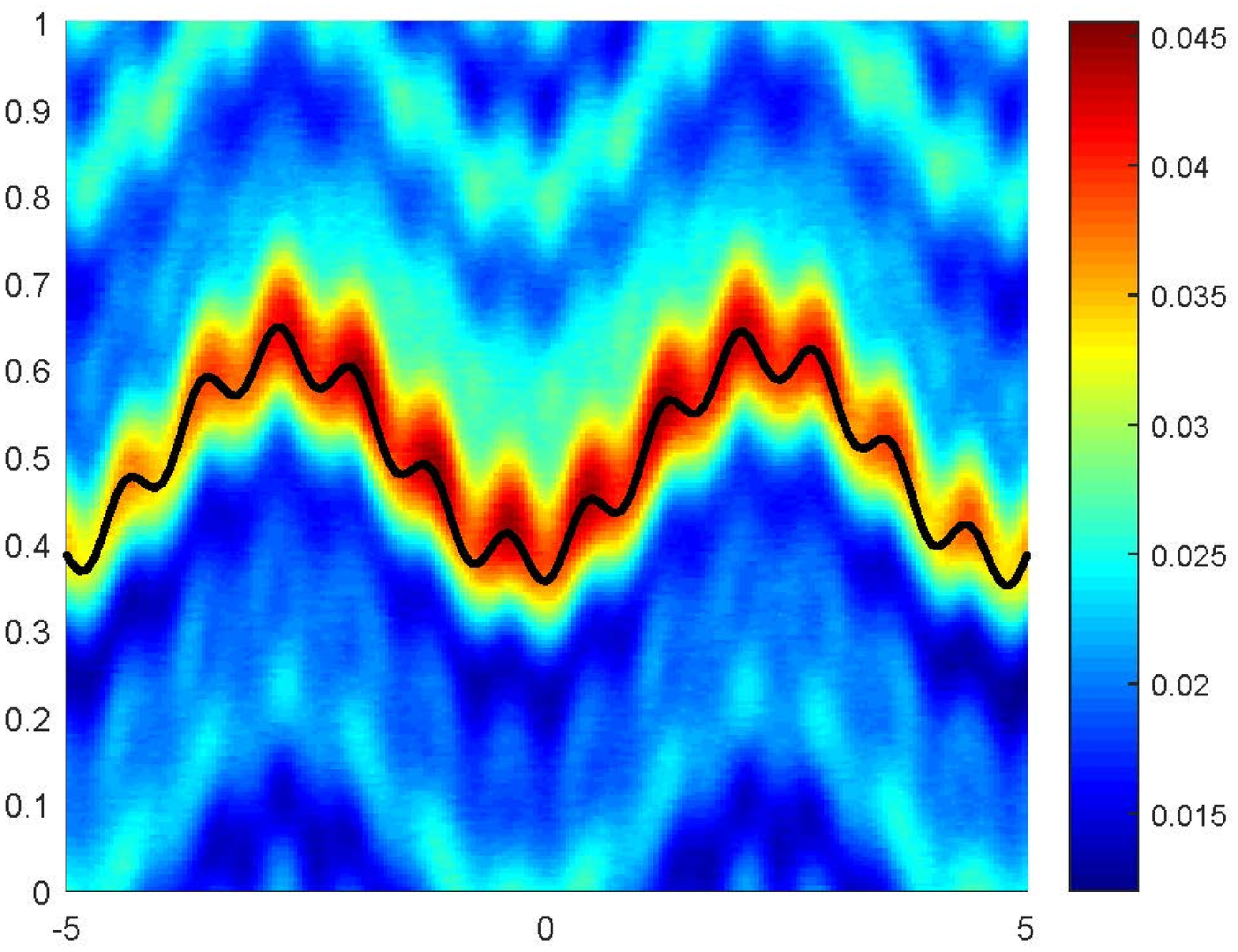}}
  \subfigure[\textbf{$k_p=11.55,\;k_s=20$}]{\includegraphics[width=1.65in]{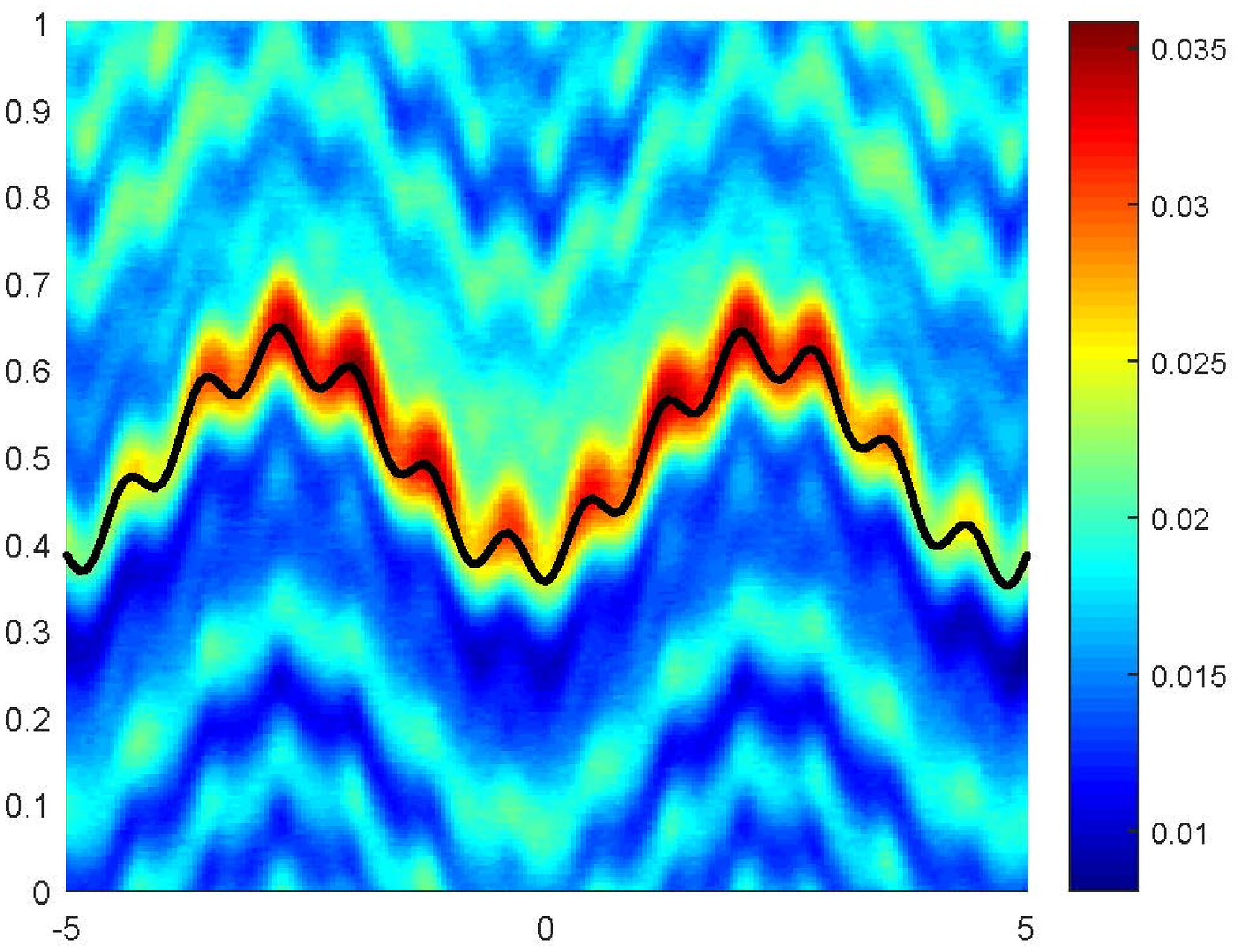}}
\caption{Reconstruction of the rough surface in Example 1 at different wave numbers.
}\label{fig1}
\end{figure}

{\bf Example 2.} This example compares the reconstructed results with different measurement places.
The exact rough surface $S$ is given by
\ben
f(x_1)= 0.5+ 0.1e^{-25(0.15x_1-0.5)^2}+0.2e^{-49(0.15x_1+0.6)^2}-0.25e^{-4x_1^2}.
\enn
The wave numbers are set to be $k_p=8.66,\;k_s=15$, and the noise level of the near-field Cauchy data
is 20\%. Fig. \ref{fig2} shows that the reconstructed result is getting better if both the measurement
line segment $T_{H,A}$ is getting closer to the rough surface and its length $2A$ is getting longer.

\begin{figure}[htbp]
  \centering
  \subfigure[\textbf{$\{(x_1,2)|\;|x_1|\leq 7.5\}$}]{
    \includegraphics[width=1.6in]{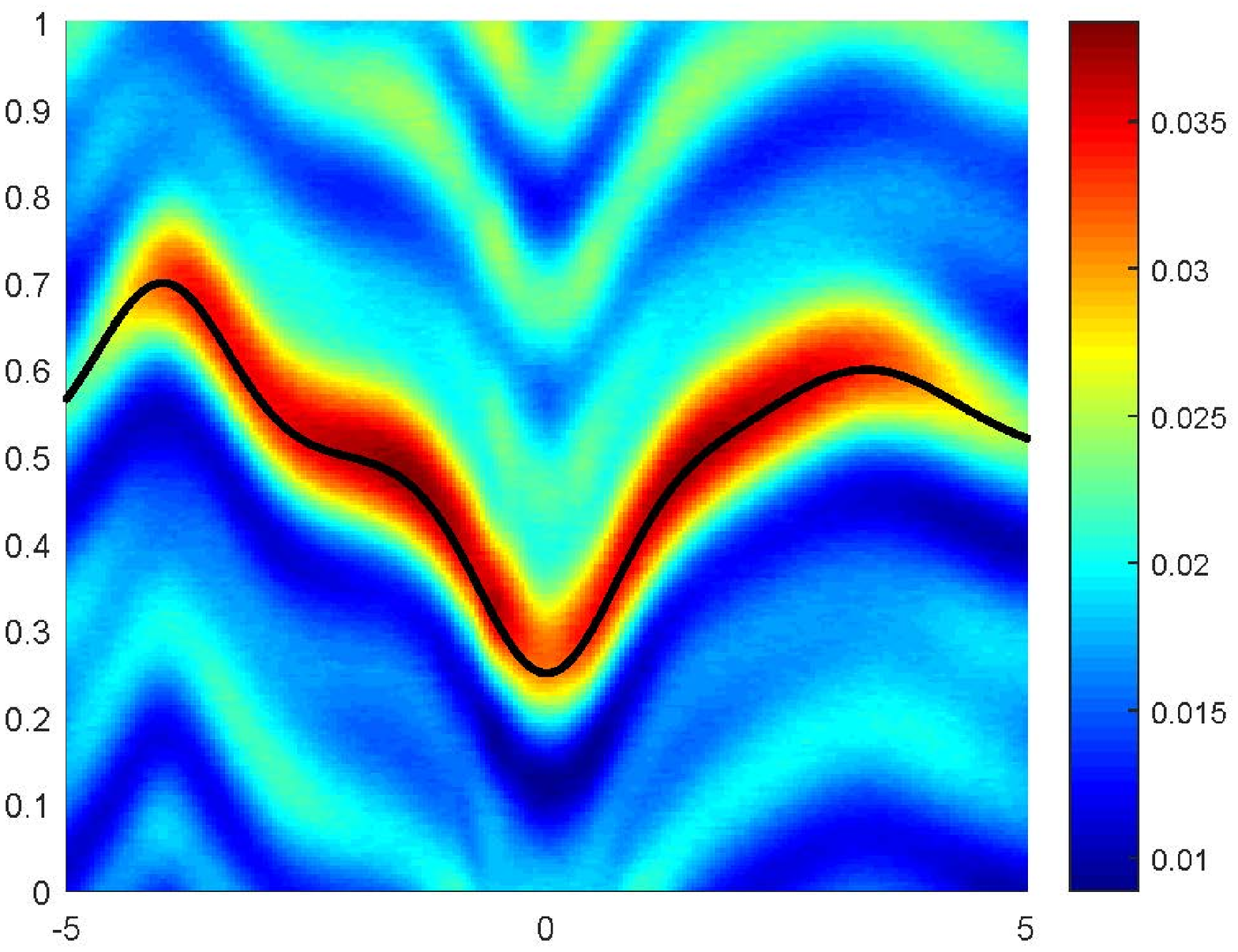}}
  \subfigure[\textbf{$\{(x_1,2)|\;|x_1|\leq 10\}$}]{
    \includegraphics[width=1.6in]{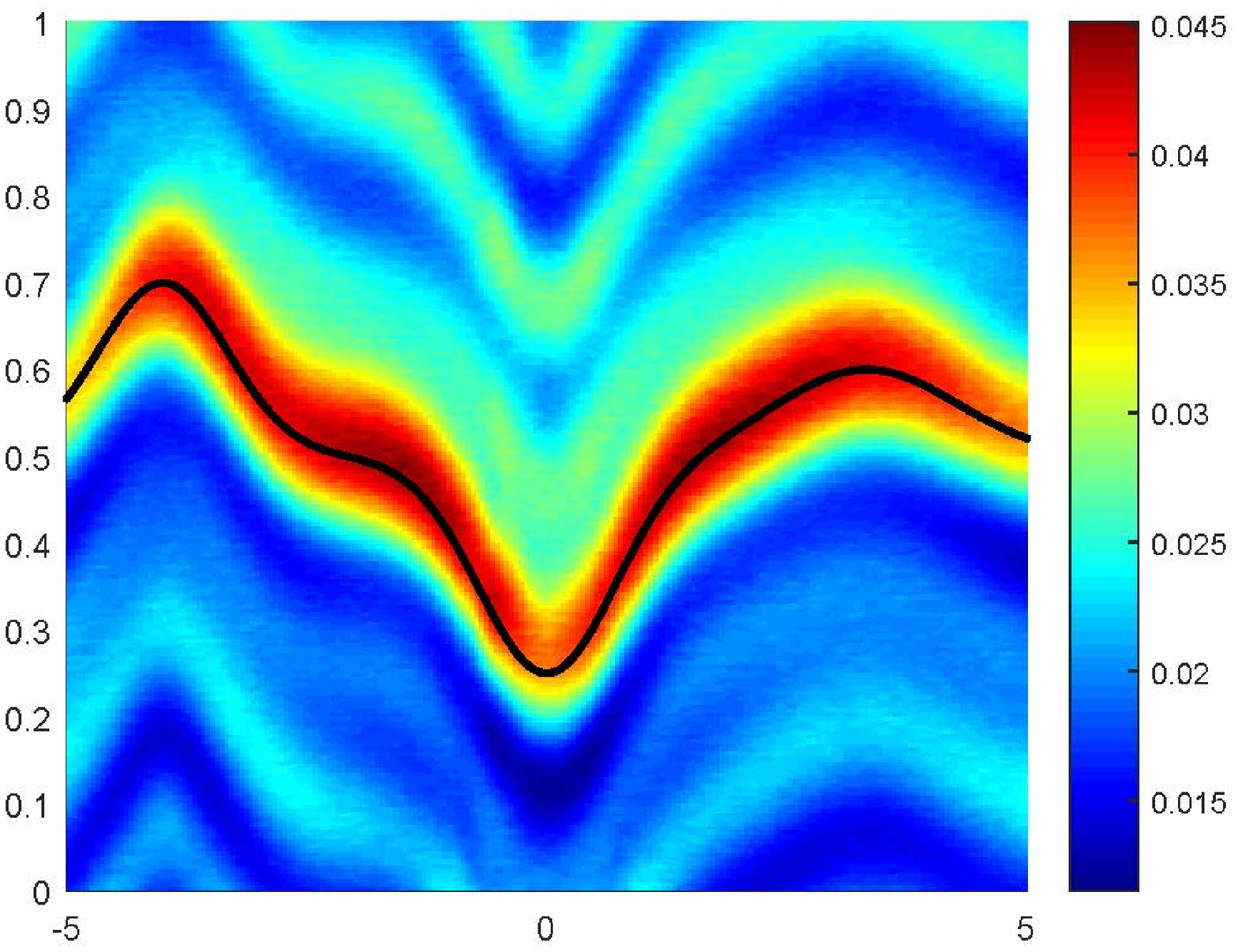}}
  \subfigure[\textbf{$\{(x_1,2)|\;|x_1|\leq 15\}$}]{
    \includegraphics[width=1.6in]{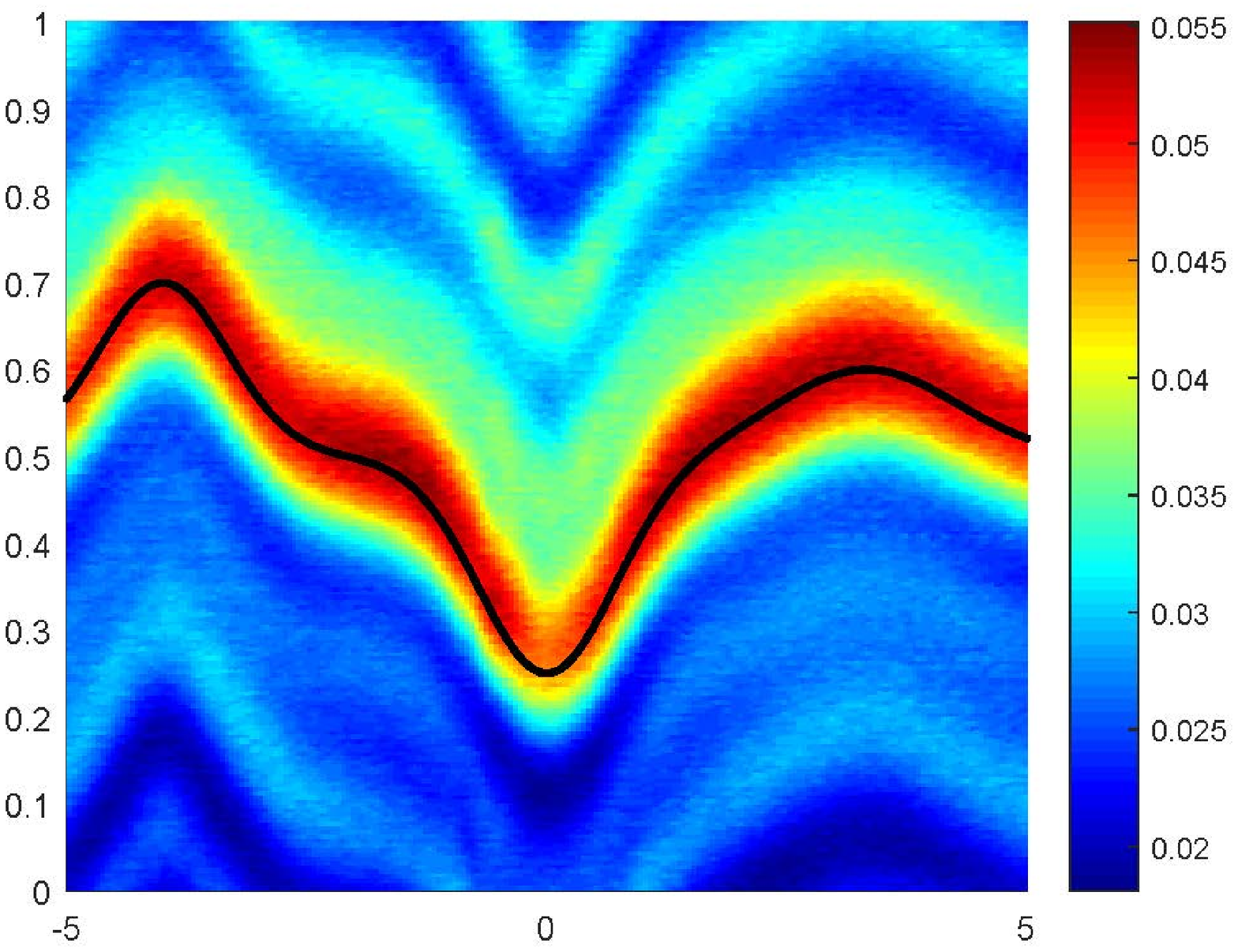}}
    \subfigure[\textbf{$\{(x_1,1.2)|\;|x_1|\leq 10\}$}]{
    \includegraphics[width=1.6in]{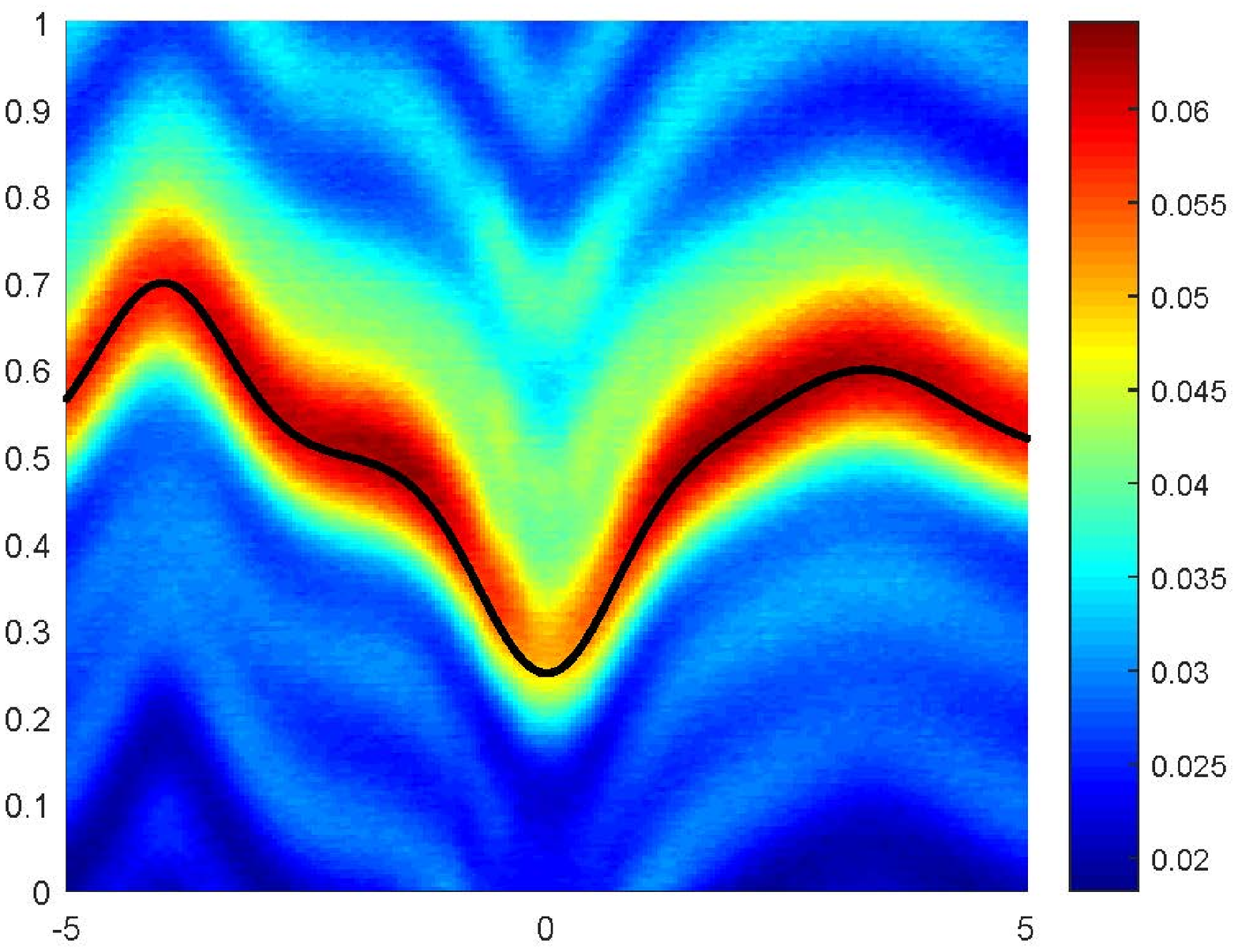}}
  \subfigure[\textbf{$\{(x_1,2)|\;|x_1|\leq 10\}$}]{
    \includegraphics[width=1.6in]{pic/example2-place/ep2-n20/e2_H2_20_n20.eps}}
  \subfigure[\textbf{$\{(x_1,4)|\;|x_1|\leq 10\}$}]{
    \includegraphics[width=1.6in]{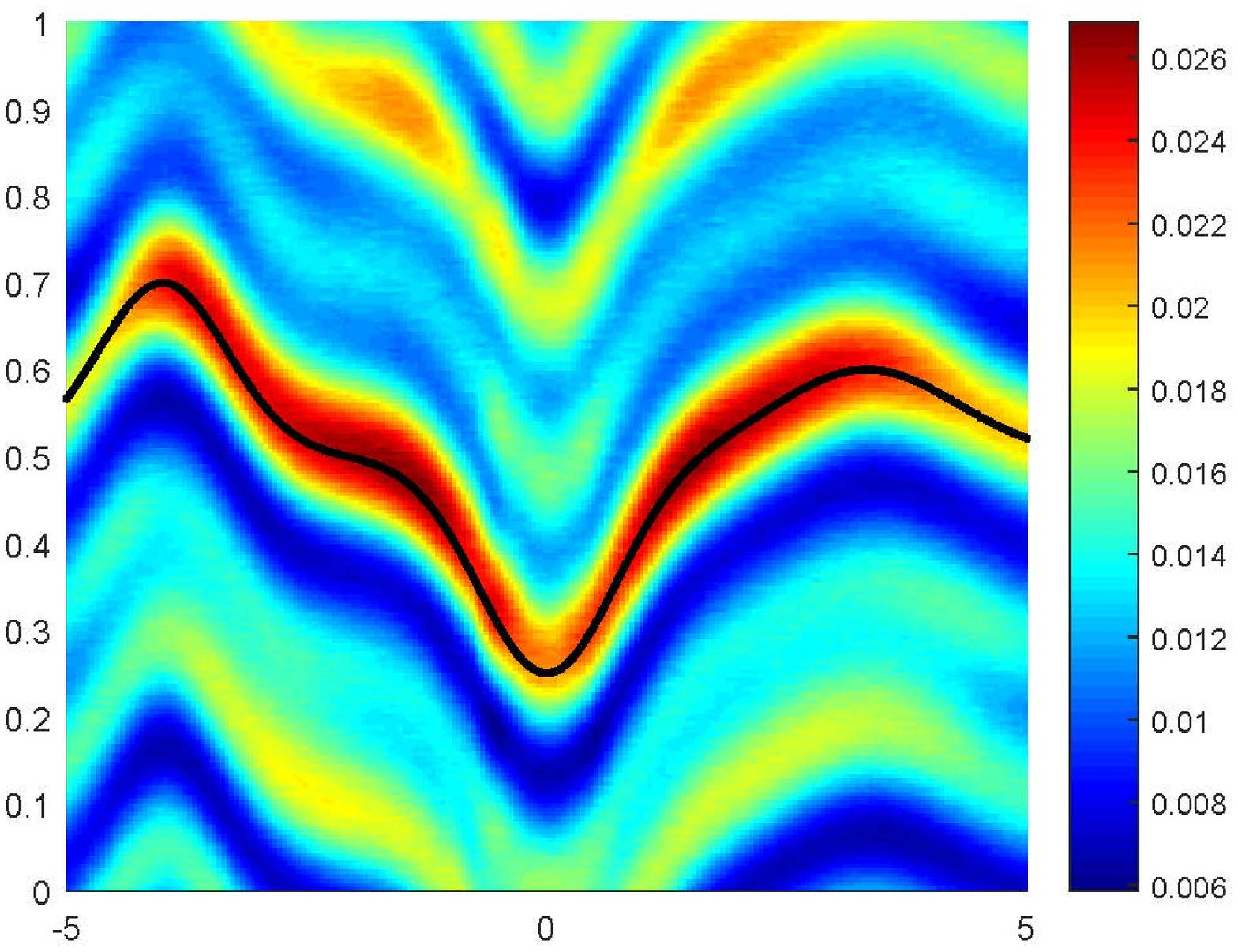}}
\caption{Reconstruction of the rough surface in Example 2 with different measurement places.
Top row (from left to right): the imaging results for the cases when the Cauchy data are
measured on $\{(x_1,2)|\;|x_1|\le7.5\}$, $\{(x_1,2)|\;|x_1|\le10\}$ and $\{(x_1,2)|\;|x_1|\le15\}$,
respectively. Bottom row (from left to right): the imaging results for the cases when the Cauchy data are
measured on $\{(x_1,1.2)|\;|x_1|\le10\}$, $\{(x_1,2)|\;|x_1|\le10\}$ and $\{(x_1,4)|\;|x_1|\leq 10\}$,
respectively.
}\label{fig2}
\end{figure}

{\bf Example 3.} This example presents the reconstructed results with different noise levels.
The exact rough surface $S$ is given by
\ben
f(x_1)= 0.5+0.1\sin(\pi x_1)+0.1\sin(0.5\pi x_1).
\enn
The near-field Cauchy data are measured on $T_{H,A}$ with $H=2,~A=20$, and the wave numbers are set to
be $k_p=8.66,\;k_s=15$. Figure \ref{fig3} presents the reconstructed results from data without noise,
with 20\% noise and with 40\% noise, respectively.
It can be seen that our imaging method is very robust to the noise in the data.

\begin{figure}[htbp]
  \centering
  \subfigure[\textbf{No noise }]{
    \includegraphics[width=1.6in]{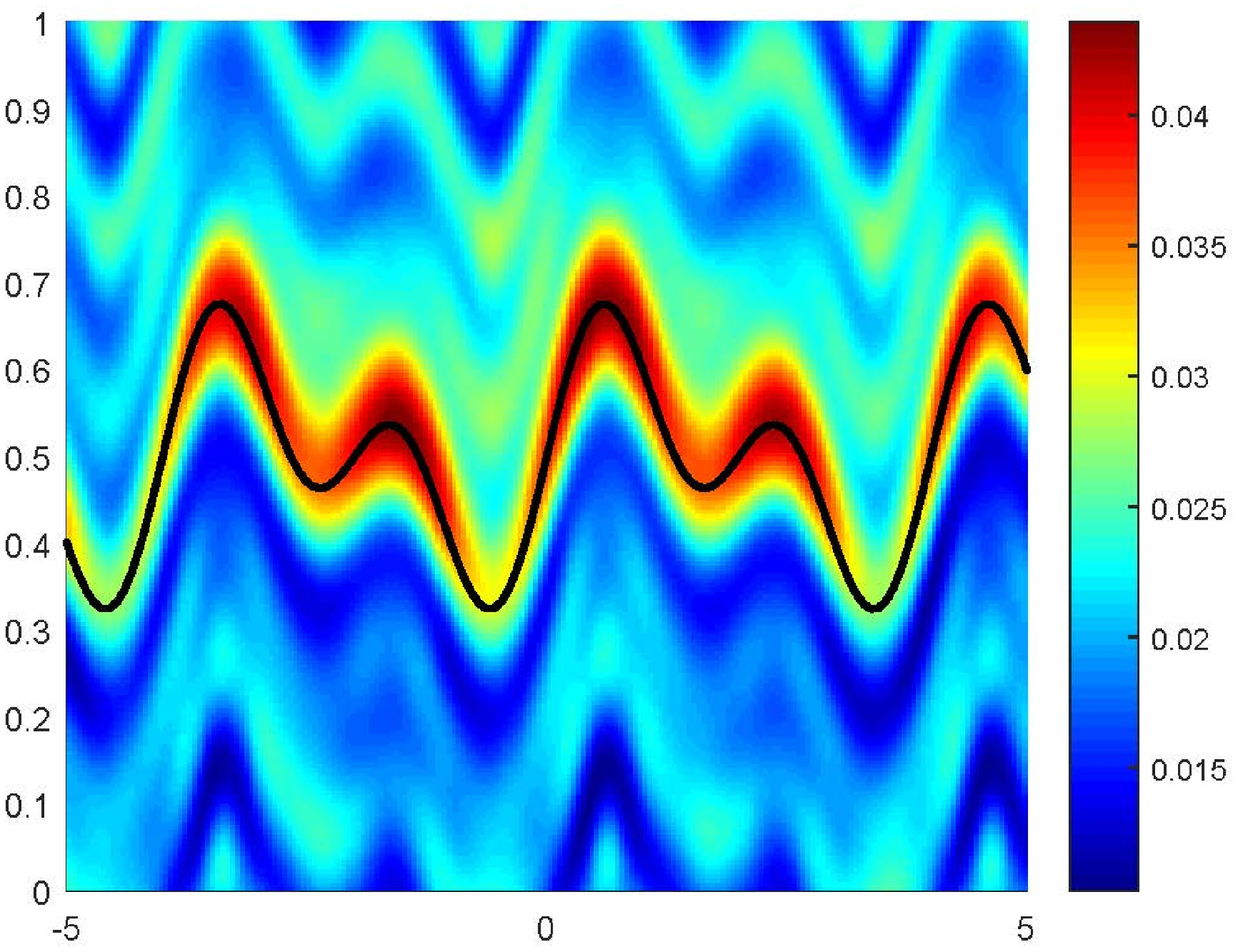}}
  \subfigure[\textbf{20\% noise }]{
    \includegraphics[width=1.6in]{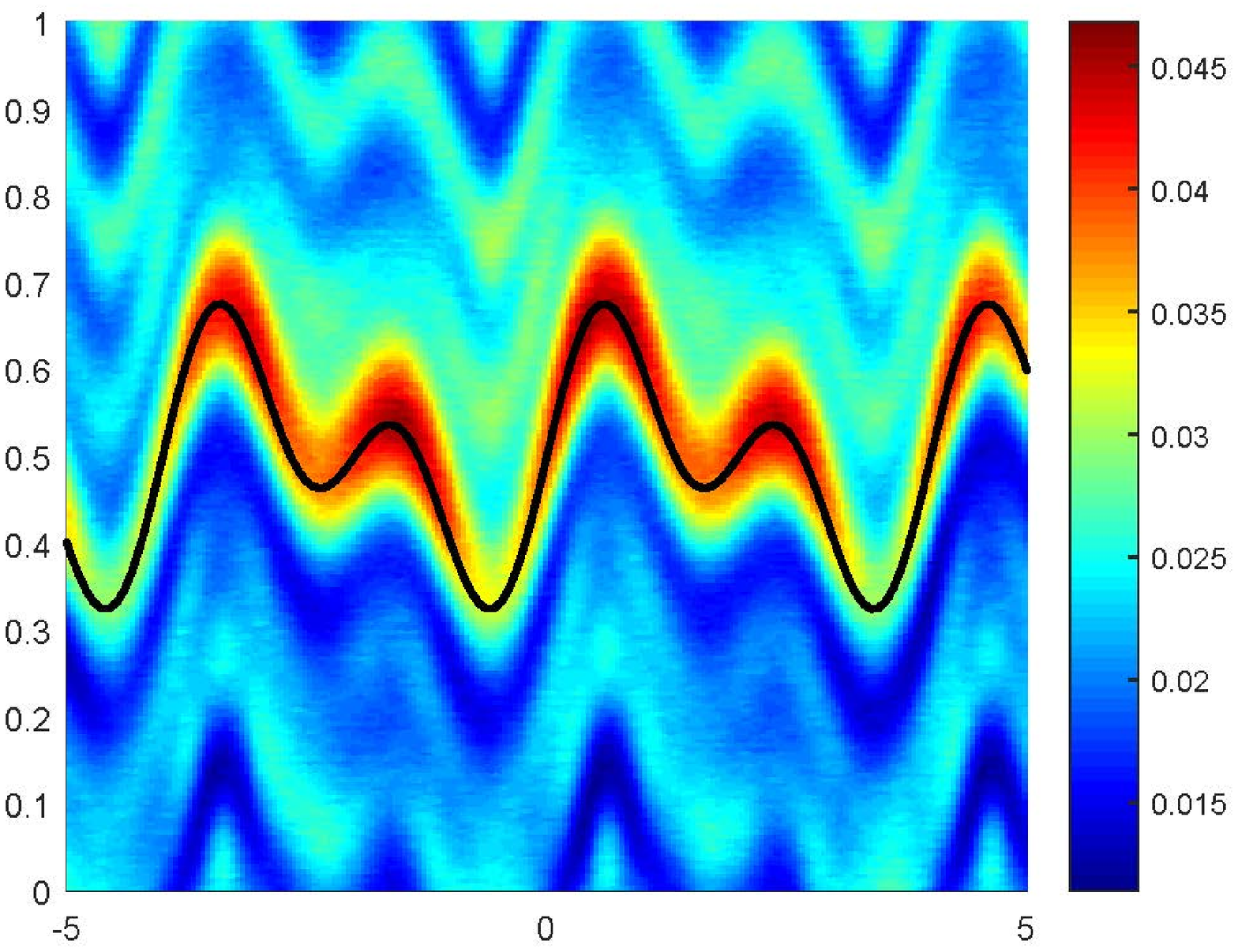}}
  \subfigure[\textbf{40\% noise }]{
    \includegraphics[width=1.6in]{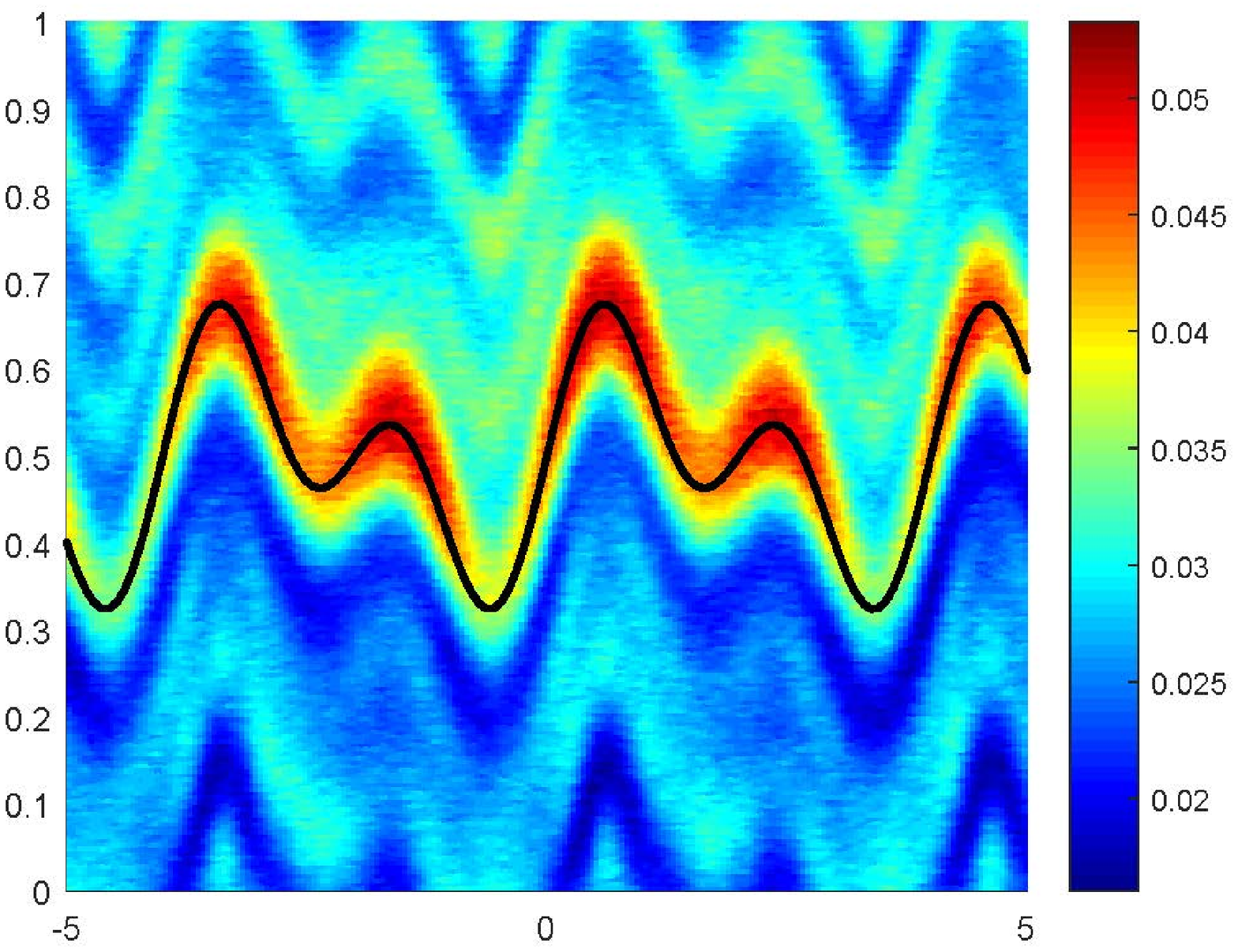}}
\caption{Imaging results of the rough surface from data with different noise levels.
}\label{fig3}
\end{figure}

The above numerical examples illustrate that the direct imaging method gives an accurate and stable
reconstruction of the unbounded rough surface and that the imaging method is robust to noise in
the measurement data.

\section{Conclusion}

This paper proposed a novel direct imaging method for inverse elastic scattering problems by an unbounded
rigid rough surface from the scattered near-field Cauchy data measured on a horizontal straight line segment
at a finite distance above the rough surface. The imaging method can reconstruct the rough surface through
computing the inner products of the measured data and the fundamental solution to the Navier equation and
thus is very fast. Numerical experiments have been carried out to show that
the reconstruction is accurate, stable and robust to noise. The imaging method can
be extended to many other cases such as inverse electromagnetic scattering problems by unbounded
rough surfaces. Progress in this direction will be reported in the future.

\section*{Acknowledgements}

This work is partly supported by the NNSF of China grants 91630309, 11501558 and 11571355.

\appendix
\renewcommand{\theequation}{\Alph{section}.\arabic{equation}}

\section*{Appendix}
\setcounter{equation}{0}

The following two lemmas are used to prove the reciprocity relation Lemma \ref{RR}.

\begin{lemma}\label{lemmaA}
Let $T_b:=\{x=(x_1,x_2)\;|\;x_1\in\R,x_2=b\}$ be a horizontal line above the rough surface $S$.
Then for the total fields ${\bm u} (z;{ x},\bm p) $ and ${\bm u} (z;{ y},\bm q) $ corresponding
to the incident point sources ${\bm u}^i(z;x,\bm p):= \G(z;x)\bm p$ and
${\bm u}^i(z;y,\bm q):=\G(z;y)\bm q$, $x,y\in\Om$
with ${\bm p}, {\bm q} \in \R^2$, we have
\be\label{eq0332}
\int_{T_b}\left[({\bm P\bm u})(z;x,{\bm p})\cdot{\bm u}(z;y,\bm q)
-({\bm P\bm u})(z;y,\bm q)\cdot {\bm u}(z;x,\bm p)\right]ds(z)=0.
\en
\end{lemma}

\begin{proof}
Define the Lam\'{e} potentials
\ben
\Psi_p:=-\frac{1}{k_p^2}\text{div}~{\bm u}\quad \text{and}\quad\Psi_s:=-\frac{1}{k_s^2}\text{div}~{\bm u}.
\enn
Then, by  \cite[Lemma 2.2]{A00} and \cite[equation(29)]{C97} we get
\be\no
{\bm u}(z_1,z_2)&=&\frac{i}{2\pi}\int_{-\infty}^{+\infty}\left(
 \begin{matrix}
 \xi\\ \sqrt{k^2_p-\xi^2}
 \end{matrix}\right)
e^{iz_1\xi+i(z_2-h)\sqrt{k_p^2-\xi^2}}\hat\Psi_{p,h}(\xi)d\xi\\ \label{eq331}
&&+\frac{i}{2\pi}\int_{-\infty}^{+\infty}\left(
 \begin{matrix}
 \sqrt{k_s^2-\xi^2}\\ -\xi
 \end{matrix}\right)
e^{iz_1\xi+i(z_2-h)\sqrt{k_s^2-\xi^2}}\hat\Psi_{s,h}(\xi)d\xi,\quad z_2>h,\;
\en
where $\hat\Psi_{p,h}(\xi):=(\mathscr{F}\Psi_{p})(\xi,h)$ and $\hat\Psi_{s,h}(\xi):=(\mathscr{F}\Psi_{s})(\xi,h)$
denote the Fourier transform of $\Psi_p$ and $\Psi_s$ on $\G_h:=\{z=(z_1,z_2)\;|\;z_1\in\R,\,z_2=h\}$
with respect to $z_1$, respectively. It follows from the angular spectrum representation (\ref{eq331}) that
\be\no
(\mathscr{F}{\bm u})(\xi,z_2)&=&i\left(
 \begin{matrix}
 \xi\\ \sqrt{k^2_p-\xi^2}
 \end{matrix}\right)
e^{ i(z_2-h)\sqrt{k_p^2-\xi^2}}\hat\Psi_{p,h}(\xi)\\ \label{fu}
&&+i\left(\begin{matrix}
 \sqrt{k_s^2-\xi^2}\\ -\xi
 \end{matrix}\right)
e^{i(z_2-h)\sqrt{k_s^2-\xi^2}}\hat\Psi_{s,h}(\xi).
\en
Using (\ref{pu}) and a straight calculation gives
\be\no
(\mathscr{F}[{\bm P\bm u}])(\xi,z_2)&=&-(\mu+\tilde{\mu})\sqrt{k_p^2-\xi^2}\left(
 \begin{matrix}
 \xi\\ \sqrt{k^2_p-\xi^2}
 \end{matrix}\right)
e^{ i(z_2-h)\sqrt{k_p^2-\xi^2}} \hat \Psi_{p,h}(\xi)\\ \no
&& -\tilde\lambda k^2_p\left(
 \begin{matrix}
 0\\ 1
 \end{matrix}\right)
e^{i(z_2-h)\sqrt{k_p^2-\xi^2}}\hat\Psi_{p,h}(\xi)\\ \no
&&-(\mu+\tilde{\mu})\sqrt{k_s^2-\xi^2}\left(
 \begin{matrix}
 \sqrt{k_s^2-\xi^2}\\ -\xi
 \end{matrix}\right)
e^{ i(z_2-h)\sqrt{k_s^2-\xi^2}}\hat \Psi_{s,h}(\xi)\\ \label{fpu}
&&+\tilde\mu k^2_s\left(
 \begin{matrix}
 1\\ 0
 \end{matrix}\right)
e^{i(z_2-h)\sqrt{k_s^2-\xi^2}}\hat\Psi_{s,h}(\xi).
\en
By Parseval's formula and the fact that $\ov{(\mathscr{F}\ol{\bm u})(\xi)}
=({\mathscr F}{\bm u})(-\xi)$, we deduce
\be\no
&&\int_{T_b}\left[({\bm P\bm u})({z};{x},\bm p)\cdot{\bm u}({z};{y},\bm q)
-({\bm P\bm u})({z};{y},\bm q)\cdot{\bm  u}({z};{x},\bm p)\right]ds(z)\\ \no
&&\quad=(2\pi)^{-1}\int_{-\infty}^{+\infty}\left[(\mathscr{F}[{\bm P\bm u}])((\xi,b);x,\bm p)
  \ov{(\mathscr{F}\ov{\bm u})}((\xi,b);y,\bm q)\right.\\ \no
&&\qquad\;\;\left.-(\mathscr{F}[{\bm P\bm u}])((\xi,b);{y},\bm q)\ov{(\mathscr{F}
   \ov{\bm u})}((\xi,b);x,\bm p)\right]d\xi\\ \no
&&\quad=(2\pi)^{-1}\int_{-\infty}^{+\infty}\left[(\mathscr{F}[{\bm P\bm u}])((\xi,b);{x},\bm p)
   (\mathscr{F}{\bm u})((-\xi,b);{y},\bm q)\right.\\ \label{314}
&&\qquad\;\;\left.-(\mathscr{F}[{\bm P\bm u}])((\xi,b);{y},\bm q)({\mathscr F}{\bm u})
  ((-\xi,b);{x},\bm p)\right]d\xi,\;\;\; b>h.
\en
Insert (\ref{fu}) and (\ref{fpu}) into (\ref{314}) to get
\ben
\int_{T_b}\left[({\bm P\bm u})({z};{x},\bm p)\cdot{\bm u}({z};{y},\bm q)
-({\bm P\bm u})({z};{y},\bm q)\cdot{\bm u}({z};{x},\bm p)\right]ds(z):=\sum_{j=1}^8 I_j,
\enn
where
\ben
I_1&=&-i(\mu+\tilde{\mu})\int_{-\infty}^{+\infty}\sqrt{k_p^2-\xi^2}(k^2_p-2\xi^2)e^{2i(b-h)\sqrt{k_p^2-\xi^2}}
      \left[\hat\Psi_{p,h,xp}(\xi)\hat\Psi_{p,h,yq}(-\xi)\right.\\
&&\quad\left.-\hat\Psi_{p,h,yq}(\xi)\hat\Psi_{p,h,xp}(-\xi)\right]d\xi,\\
I_2&=&-i(\mu+\tilde{\mu})\int_{-\infty}^{+\infty}\sqrt{k_s^2-\xi^2}(k^2_s-2\xi^2)e^{2i(b-h)\sqrt{k_s^2-\xi^2}}
      \left[\hat\Psi_{s,h,xp}(\xi)\hat\Psi_{s,h,yq}(-\xi)\right.\\
&&\quad\left.-\hat\Psi_{s,h,yq}(\xi)\hat\Psi_{s,h,xp}(-\xi)\right]d\xi,\\
I_3&=&-i(\mu+\tilde{\mu})\int_{-\infty}^{+\infty}\xi\sqrt{k_p^2-\xi^2}(\sqrt{k^2_p-\xi^2}+\sqrt{k^2_s-\xi^2})
      e^{i(b-h)(\sqrt{k^2_p-\xi^2}+\sqrt{k^2_s-\xi^2})}\\
&&\qquad\;\cdot\left[\hat\Psi_{p,h,xp}(\xi)\hat\Psi_{s,h,yq}(-\xi)
        -\hat\Psi_{p,h,yq}(\xi)\hat\Psi_{s,h,xp}(-\xi)\right]d\xi,\\
I_4&=& i(\mu+\tilde{\mu})\int_{-\infty}^{+\infty}\xi\sqrt{k_s^2-\xi^2}(\sqrt{k^2_p-\xi^2}+\sqrt{k^2_s-\xi^2})
      e^{i(b-h)(\sqrt{k^2_p-\xi^2}+\sqrt{k^2_s-\xi^2})}\\
&&\qquad\;\cdot\left[\hat\Psi_{s,h,xp}(\xi)\hat\Psi_{p,h,yq}(-\xi)
       -\hat\Psi_{s,h,yq}(\xi)\hat\Psi_{p,h,xp}(-\xi)\right]d\xi,\\
I_5&=& -i\tilde\lambda k^2_p\int_{-\infty}^{+\infty}\sqrt{k_p^2-\xi^2}e^{2i(b-h)\sqrt{k_p^2-\xi^2}}\\
&&\qquad\;\cdot\left[\hat\Psi_{p,h,xp}(\xi)\hat\Psi_{p,h,yq}(-\xi)
        -\hat\Psi_{p,h,yq}(\xi)\hat\Psi_{p,h,xp}(-\xi)\right]d\xi,\\
I_6&=&i\tilde\mu k^2_s\int_{-\infty}^{+\infty}\sqrt{k_s^2-\xi^2}e^{2i(b-h)\sqrt{k_s^2-\xi^2}}\\
&&\qquad\;\cdot\left[\hat\Psi_{s,h,xp}(\xi)\hat\Psi_{s,h,yq}(-\xi)
        -\hat\Psi_{s,h,yq}(\xi)\hat\Psi_{s,h,xp}(-\xi)\right]d\xi,\\
I_7&=&-i\tilde\lambda k^2_p\int_{-\infty}^{+\infty}\xi e^{i(b-h)
       \left(\sqrt{k^2_p-\xi^2}+\sqrt{k^2_s-\xi^2}\right)}\\
&&\qquad\;\cdot\left[\hat\Psi_{p,h,xp}(\xi)\hat\Psi_{s,h,yq}(-\xi)
       -\hat\Psi_{p,h,yq}(\xi)\hat\Psi_{s,h,xp}(-\xi)\right]d\xi,\\
I_8&=&-i\tilde\mu k^2_s\int_{-\infty}^{+\infty}\xi e^{i(b-h)
       \left(\sqrt{k^2_p-\xi^2}+\sqrt{k^2_s-\xi^2}\right)}\\
&&\qquad\;\cdot\left[\hat\Psi_{s,h,xp}(\xi)\hat\Psi_{p,h,yq}(-\xi)
       -\hat\Psi_{s,h,yq}(\xi)\hat\Psi_{p,h,xp}(-\xi)\right]d\xi.
\enn
From the fact that
\ben
\int_{-\infty}^{+\infty}a(\xi)f(\xi)g(-\xi)d\xi &=&\int_{-\infty}^{+\infty}a(\xi)f(-\xi)g(\xi)d\xi\quad
\mbox{if $a$ is an even function},\\
\int_{-\infty}^{+\infty}b(\xi)f(\xi)g(-\xi)d\xi &=&-\int_{-\infty}^{+\infty}b(\xi)f(-\xi)g(\xi)d\xi\quad
\mbox{if $b$ is an odd function},
\enn
we obtain that $I_1=I_2=I_5=I_6=0$ and
\ben
I_3+I_4+I_7+I_8&=&\int_{-\infty}^{+\infty}i[(\mu+\tilde\mu)(k_s^2-k_p^2)-\tilde\lambda k_p^2-\tilde\mu k^2_s]
\xi e^{i(b-h)(\sqrt{k^2_p-\xi^2}+\sqrt{k^2_s-\xi^2})}\\
&&\qquad\;\cdot\left[\hat\Psi_{s,h,xp}(\xi)\hat\Psi_{p,h,yq}(-\xi)-\hat\Psi_{s,h,yq}(\xi)\hat\Psi_{p,h,xp}(-\xi)\right]d\xi.
\enn
Since $\mu+\la=\tilde\mu+\tilde\la$, $k_s^2=\om^2/\mu$ and $k^2_p=\om^2/(2\mu+\la)$, then $I_3+I_4+I_7+I_8=0$.
The required equation (\ref{eq0332}) thus follows. The lemma is thus proved.
\end{proof}

\begin{lemma}\label{lemmaB}
Let $B_{\vep}({x}):=\{z\in\R^2\;|\;|z-x|<\vep\}$ be a ball centered at $x$ with radius $\vep$.
Then, for any total field $\bm u$ of the scattering problem corresponding to the incident point source
${\bm u}^i(z;x,\bm p)=\G(z,x)\bm p$ with ${\bm p}\in \R^2$ we have
\be\label{BB}
{\bm u}(x)\cdot{\bm p}=\int_{\pa B_{\vep}(x)}\left[({\bm P\bm u})(z)\cdot{\bm u}^i(z;x,\bm p)
-{\bm u}(z)\cdot({\bm P\bm u}^i)(z;x,\bm p)\right]ds(z).
\en
\end{lemma}

\begin{proof}
By a direct calculation it can be derived that on $\pa B_{\vep}(x)$,
\ben
\Gamma(z,x):=I_1+I_2+I_3+I_4,
\enn
where
\ben
I_1&=&\frac{i}{4\om^2}\left(k_p^2H^{(1)}_0(k_p\vep)-k_s^2H^{(1)}_0(k_s\vep)\right)\hat{R}\otimes\hat{R},\\
I_2&=&-\frac{i}{2\om^2\vep}\left(k_p H^{(1)}_1(k_p\vep)-k_s H^{(1)}_1(k_s\vep)\right)\hat{R}\otimes\hat{R},\\
I_3&=&\frac{i}{4\om^2\vep}\left(k_p H^{(1)}_1(k_p\vep)-k_s H^{(1)}_1(k_s\vep)\right){I},\\
I_4&=&\frac{i}{4\mu}H^{(1)}_0(k_s\vep){I}.
\enn
Here, $R=z-x:=(R_1,R_2),~|z-x|=\vep$ and $\hat{R}=(z-x)/|z-x|$.

By the definition of the generalised stress operator ${\bm P}$ in (\ref{pu}),
and setting $R_{\perp}=(R_2,-R_1)$, we have
\ben
{\bm P}(I_1\bm q)&=&\frac{i}{4\om^2}\left[(\la+2\mu)\left(k_p^3H^{(1)'}_0(k_p\vep)-k_s^3H^{(1)'}_0(k_s\vep)\right)\right.\\
&&\qquad\;\;\left.+\frac{\la}{\vep}\left(k_p^2H^{(1)}_0(k_p\vep)
            -k_s^2 H^{(1)}_0(k_s\vep)\right)\right]\hat{R}\otimes\hat{R}\bm  q\\
&&\qquad\;\;+\frac{i}{4\omega^2}\frac{\mu}{\vep}\left(k_p^2 H^{(1)}_0(k_p\vep)
            -k_s^2 H^{(1)}_0(k_s\vep)\right)\hat{R_{\perp}}\otimes\hat{R_{\perp}}\bm  q ,\\
{\bm P}(I_2\bm q)&=&-\frac{i}{2\omega^2}\left[\frac{\la+2\mu}{\vep}\left(k_p^2H^{(1)'}_1(k_p\vep)
-k_s^2H^{(1)'}_1(k_s\vep)\right)\right.\\
&&\qquad\;\;\left.-\frac{2\mu}{\vep^2}\left(k_p H^{(1)}_1(k_p\vep)
            -k_s H^{(1)}_1(k_s\vep)\right)\right]\hat{R}\otimes\hat{R}\bm  q\\
&&\qquad\;\;-\frac{i}{2\omega^2}\frac{\mu}{\vep^2}\left(k_p H^{(1)}_1(k_p\vep)
            -k_s H^{(1)}_1(k_s\vep)\right)\hat{R_{\perp}}\otimes\hat{R_{\perp}}\bm  q ,\\
{\bm P}(I_3\bm q)&=&\frac{i}{4\omega^2}\left[\left(\frac{k_p^2}{\vep}H^{(1)'}_1(k_p\vep)
         -\frac{k_p}{\vep^2}H^{(1)}_1(k_p\vep)\right)\right.\\
&&\qquad\;\;\;\left.-\left(\frac{k_s^2}{\vep}H^{(1)'}_1(k_s\vep)
  -\frac{k_s}{\vep^2}H^{(1)}_1(k_s\vep)\right)\right]\left[(\la+\mu)\hat{R}\otimes\hat{R}+\mu I\right]\bm q,\\
{\bm P}(I_4\bm q)&=&\frac{i}{4\mu}k_sH^{(1)'}_0(k_s\vep)\left[(\la+\mu)\hat{R}\otimes\hat{R}+\mu I\right]\bm q.
\enn
Thus it follows that
\ben
{\bm P}(\G(z,{x})\bm q)&:=&-\frac{i}{4}k_sH^{(1)}_1(k_s\vep)I
  +\frac{i\mu}{4\om^2\vep}\left[\frac{2}{\vep}\left(k_p H^{(1)}_1(k_p\vep)-k_s H^{(1)}_1(k_s\vep)\right)\right.\\
&&\qquad\;\left.-\left(k_p^2 H^{(1)}_0(k_p\vep)-k_s^2 H^{(1)}_0(k_s\vep)\right)\right]
(\hat{R_{\perp}}\otimes\hat{R_{\perp}}-\hat{R}\otimes\hat{R})\bm q.
\enn
Noting that
\ben
\frac{i}{4}H^{(1)}_0(k_\beta\vep) &\approx& \frac{1}{2\pi}\ln\frac{1}{\vep}+\frac{i}{4}
-\frac{1}{2\pi}\ln\frac{k_\beta}{2}-\frac{C}{2\pi}\quad\mbox{as}\;\vep\rightarrow 0,\\
H^{(1)}_m(k_\beta\vep) &\approx& \frac{(k_\beta\vep)^m}{2^m m!}-i\frac{2^m(m-1)!}{\pi(k_\beta\vep)^m},\quad m=1,2,
\quad\mbox{as}\;\vep\rightarrow 0,
\enn
where $C$ denotes Euler's constant, we deduce that
\ben\no
\lim_{\vep\rightarrow 0}2\pi\vep\G({{z},x})&=&\lim_{\vep\rightarrow 0}\vep
\left[\frac{k_p^2}{\omega^2}\left(\ln\frac{1}{\vep}-\ln\frac{k_p}{2}-C\right)
-\frac{k_s^2}{\omega^2}\left(\ln\frac{1}{\vep}-\ln\frac{k_s}{2}-C\right)\right]\hat{R}\otimes\hat{R}\\ \no
&&+\lim_{\vep\rightarrow 0}\vep\left[\frac{i\pi(k_p^2-k_s^2)}{4\omega^2}+
  \frac{k_s^2}{\omega^2}\left(\ln\frac{1}{\vep}+\frac{i\pi}{2}-\ln\frac{k_s}{2}-C\right)\right]I=0,
\enn
\ben\no
\lim_{\vep\rightarrow 0}2\pi\vep{\bm P}(\G({z},{x})\bm q)&=&\lim_{\vep\rightarrow 0}
\left[\left(-\frac{i\pi k_s^2\vep^2}{4}+1\right)I
+\frac{i\pi\mu\vep^2(k_p^2-k_s^2)}{8\om^2}(\hat{R_{\perp}}\otimes\hat{R_{\perp}}
-\hat{R}\otimes\hat{R})\bm q\right]\\ \label{02}
&=&I.
\enn
The above two equations and the mean value theorem then imply the required equation (\ref{BB}).
\end{proof}

\end{document}